\documentclass[11pt]{article}
  \def\SvZfontcode{8} 
  \def\SvZslantedGreekCapitals{1}
  \usepackage{amsmath,amsthm,amscd}

\usepackage[T1]{fontenc}

\def\SvZrequireslantedRedef{0}
\ifcase\SvZfontcode
\usepackage{bm}
\usepackage{amssymb}
\def\SvZrequireslantedRedef{1}
\or
\usepackage{fouriernc}
\usepackage{bm}
\def\SvZrequireslantedRedef{1}
\or
\usepackage{fourier}
\usepackage{bm}
\def\SvZrequireslantedRedef{1}
\or
\usepackage{amssymb}
\if\SvZslantedGreekCapitals 1
\usepackage[slantedGreek]{mathptmx}
\else
\usepackage{mathptmx}
\fi
\DeclareMathAlphabet{\bm}{OT1}{ptm}{b}{it} 
\or
\usepackage{kmath,kerkis}
\usepackage{bm}
\def\SvZrequireslantedRedef{1}
\or
\if\SvZslantedGreekCapitals 1
\usepackage[sc,slantedGreek]{mathpazo}
\else
\usepackage[sc]{mathpazo}
\fi
\usepackage{bm}
\or
\usepackage[adobe-utopia]{mathdesign}
\usepackage{bm}
\or
\usepackage[urw-garamond]{mathdesign}
\usepackage{bm}
\or
\usepackage[bitstream-charter]{mathdesign}
\usepackage{bm}
\or
\usepackage{lmodern}
\usepackage{bm}
\usepackage{amssymb}
\def\SvZrequireslantedRedef{1}
\or
\usepackage{arev}
\fi

\if\SvZrequireslantedRedef 1
\if\SvZslantedGreekCapitals 1
\renewcommand{\Gamma}{\varGamma}
\renewcommand{\Delta}{\varDelta}
\renewcommand{\Theta}{\varTheta}
\renewcommand{\Lambda}{\varLambda}
\renewcommand{\Xi}{\varXi}
\renewcommand{\Pi}{\varPi}
\renewcommand{\Sigma}{\varSigma}
\renewcommand{\Upsilon}{\varUpsilon}
\renewcommand{\Phi}{\varPhi}
\renewcommand{\Psi}{\varPsi}
\renewcommand{\Omega}{\varOmega}
\fi
\fi

\renewcommand{\phi}{\varphi}

\reversemarginpar

\usepackage{url}
\usepackage{listings}
\usepackage{kbordermatrix}
\usepackage{booktabs}
\usepackage{graphicx}
\usepackage{color}

\ifx\SvZinPresentation\undefined
\usepackage{float}
\usepackage{longtable}
\usepackage[margin=10pt,labelfont=bf,textfont=it]{caption}

\usepackage[colorlinks]{hyperref}

\fi

\newcommand{\mathds}{\mathbb}

\DeclareMathOperator{\rank}{rk}

\DeclareMathOperator{\si}{si}
\DeclareMathOperator{\co}{co}

\newcommand{\N}{\mathds{N}}
\newcommand{\field}{\mathds{F}}
\newcommand{\group}{G}

\DeclareMathOperator{\GF}{GF}
\newcommand{\ring}{R}
\newcommand{\parf}{\ensuremath{\mathbb{P}}} 
\newcommand{\pf}{\parf}

\newcommand{\nreg}{\uniform_1}
\newcommand{\psru}{\mathds{S}}
\newcommand{\sru}{\ensuremath{\sqrt[6]{1}}}

\newcommand{\uniform}{\mathds{U}}

\newcommand{\ignore}[1]{}

\DeclareMathOperator{\AG}{AG} 

\newcommand{\whirl}[1]{\mathcal{W}^{#1}}

\let\Oldsetminus\setminus
\renewcommand{\setminus}{\ensuremath{-}}

\newcommand{\delete}{\ensuremath{\!\Oldsetminus\!}}
\newcommand{\contract}{\ensuremath{\!/}}
\newcommand{\symdiff}{\triangle}
\newcommand{\minorof}{\ensuremath{\preceq}}

\DeclareMathOperator{\bw}{bw}

\newcommand{\delset}{\mathbf{D}}
\newcommand{\conset}{\mathbf{C}}
\newcommand{\essset}{\mathbf{E}}

\DeclareMathOperator{\matset}{\mathcal{M}}

\newcommand{\bip}{G}

\ifx\SvZinPresentation\undefined
\newtheorem{theorem}{Theorem}[section]
\else
\newtheorem{theorem}{Theorem}
\fi

\newtheorem{lemma}[theorem]{Lemma}
\newtheorem{proposition}[theorem]{Proposition}
\newtheorem{definition}[theorem]{Definition}
\newtheorem{corollary}[theorem]{Corollary}

\newtheorem{claim}{Claim}[theorem] 
\newtheorem{assumption}[claim]{Assumption} 

\newenvironment{claimenv}{\list{}{\rightmargin0pt\leftmargin10pt\topsep0pt}\item[]}{\endlist}

\newenvironment{subproof}{\begin{claimenv}\begin{proof}}{\end{proof}\end{claimenv}}
\newtheorem{conjecture}[theorem]{Conjecture}

\usepackage[numbers]{natbib}
\newcommand{\Dutchvon}[2]{#2}

\newcommand{\sepbound}[1]{\ensuremath{2^{#1+1}}}

\begin{document}
\title{Stability, fragility, and Rota's Conjecture\footnote{Parts of this paper were previously published in the third author's PhD thesis \cite{vZ09}. The research of all authors was partially supported by a grant from the Marsden Fund of New Zealand. The first author was also supported by a FRST Science \& Technology post-doctoral fellowship. The third author was also supported by the Netherlands Organisation for Scientific Research (NWO).}}
\author{Dillon Mayhew\thanks{School of Mathematics, Statistics and Operations Research, Victoria University of Wellington, New Zealand. E-mail: \url{Dillon.Mayhew@msor.vuw.ac.nz}, \url{Geoff.Whittle@msor.vuw.ac.nz}} \and Geoff Whittle\footnotemark[2]  \and Stefan H. M. van Zwam\thanks{Centrum Wiskunde en Informatica, Postbus 94079, 1090 GB Amsterdam, The Netherlands. E-mail: \url{Stefan.van.Zwam@cwi.nl}}}

\maketitle


\abstract{
Fix a matroid $N$. A matroid $M$ is $N$-\emph{fragile} if, for each element $e$ of $M$, at least one of $M\delete e$ and $M\contract e$ has no $N$-minor. The Bounded Canopy Conjecture is that all $\GF(q)$-representable matroids $M$ that have an $N$-minor and are $N$-fragile have branch width bounded by a constant depending only on $q$ and $N$.

A matroid $N$ \emph{stabilizes} a class of matroids over a field $\field$ if, for every matroid $M$ in the class with an $N$-minor, every $\field$-representation of $N$ extends to at most one $\field$-representation of $M$.

We prove that, if Rota's conjecture is false for $\GF(q)$, then either the Bounded Canopy Conjecture is false for $\GF(q)$ or there is an infinite chain of $\GF(q)$-representable matroids, each not stabilized by the previous, each of which can be extended to an excluded minor.

Our result implies the previously known result that Rota's Conjecture holds for $\GF(4)$, and that the classes of near-regular and sixth-roots-of-unity have a finite number of excluded minors. However, the bound that we obtain on the size of such excluded minors is considerably larger than that obtained in previous proofs. For $\GF(5)$ we show that Rota's Conjecture reduces to the Bounded Canopy Conjecture.
}


\section{Introduction}\label{sec:intro}
Rota's Conjecture, widely regarded as the most important open problem in matroid theory, is as follows.

\begin{conjecture}[\citet{Rot71}]\label{con:rota1}
  For all prime powers $q$, the class of matroids representable over $\GF(q)$ can be characterized by a finite set of excluded minors.
\end{conjecture}

Progress on this conjecture has been intermittent. It has been settled completely only for $q \leq 4$ \citep{Tut65,Bix79,Sey79,GGK}. \citet{GGW06} showed that an excluded minor contains no large projective geometry. Another partial result towards Rota's Conjecture is the following:

\begin{theorem}[\citet{GW02}]\label{thm:rotabw}
  Let $\field$ be a finite field and $k \in \N$. Let $\mathcal{M}$ be a minor-closed class of $\field$-representable matroids. Then finitely many excluded minors for $\mathcal{M}$ have branch width $k$.
\end{theorem}

In 1996, \citet{SW96} introduced matroids representable over \emph{partial fields}. Anticipating some of the definitions in Section~\ref{sec:partialfields}, we say a partial field $\parf$ is \emph{finitary} if there exists a homomorphism $\phi:\parf\rightarrow\GF(q)$ for some prime power $q$. We denote by $\matset(\parf)$ the set of $\parf$-representable matroids. Since homomorphisms preserve representability, $\matset(\parf)\subseteq\matset(\GF(q))$ for some prime power $q$ if $\pf$ is finitary. Conjecture~\ref{con:rota1} can then be generalized as follows:

\begin{conjecture}\label{con:rota2}
  For every finitary partial field $\parf$, $\matset(\parf)$ can be characterized by a finite set of excluded minors.
\end{conjecture}

Like Rota's Conjecture, this conjecture has been settled for only a handful of partial fields. In particular, it is known for the regular, sixth-roots-of-unity, and near-regular partial fields  \citep{Tut65,GGK,HMZ11}.

At the moment Geelen, Gerards, and Whittle are carrying out a project aimed at proving that $\matset(\GF(q))$ is well-quasi-ordered with respect to the minor-order (see, for instance, \citet{GGW06b}). That result, when combined with a proof of Conjecture~\ref{con:rota1}, would imply Conjecture~\ref{con:rota2}, since proper minor-closed classes of $\matset(\GF(q))$ would be characterized by a finite set of excluded minors. In this paper we set the stage for a proof of Rota's Conjecture for $q = 5$, by reducing it to a conjecture that should be a consequence of the structure theory being developed for the matroid minors project.

To state our main result we need to introduce a few concepts. We say that a matroid $N$ \emph{stabilizes} a matroid $M$ over a partial field $\pf$ if, for each minor $M'$ of $M$ isomorphic to $N$, each $\pf$-representation of $M'$ extends to at most one $\pf$-representation of $M$. A matroid $N$ is a \emph{stabilizer} for a class of matroids $\matset$ if $N$ stabilizes each 3-connected member of $\matset$. We will be more precise in Definition \ref{def:stab}. Stabilizers were introduced by Whittle \cite{Whi96b}, who proved that checking if a matroid is a stabilizer requires a finite amount of work.

A second concept we need is \emph{fragility}. Let $N$, $M$ be matroids. Then $M$ is \emph{$N$-fragile} if, for all $e \in E(M)$, at least one of $M\delete e, M\contract e$ has no minor isomorphic to $N$. If $M$ is $N$-fragile and $N$ is a minor of $M$ then $M$ is \emph{strictly $N$-fragile}. A slightly more general definition will be given in Section~\ref{sec:fragility}. Note that fragility has been studied previously under a different name. If $\mathcal{M}$ is a minor-closed class of matroids, then a matroid $M$ is \emph{almost-$\mathcal{M}$} if, for each $e \in E(M)$, at least one of $M\delete e$ and $M\contract e$ is in $\mathcal{M}$. See, for instance, \cite{Oxl90, KL02}. 

A third concept, already mentioned in Theorem \ref{thm:rotabw}, is branch width. Roughly speaking, a matroid with high branch width cannot be decomposed into small pieces along low-order separations. It is closely related to the notion of tree width in graphs. We will define the branch width of a matroid, denoted by $\bw(M)$, in Section~\ref{sec:conn}.

\begin{definition}
  Let $\matset$ be a class of matroids. Then $N$ has \emph{bounded canopy} over $\matset$ if there exists an integer $l$ such that, for all strictly $N$-fragile matroids $M \in \matset$, $\bw(M) \leq l$.
\end{definition}

Finally,

\begin{definition}\label{def:well-closed}
  A class of matroids is \emph{well-closed} if it is closed under isomorphism, duality, taking minors, direct sums, and 2-sums. 	
\end{definition}

Our main result now is the following:

\begin{theorem}\label{thm:rotapf}
  Let $\parf$ be a finitary partial field, let $\matset$ be a well-closed class of $\pf$-representable matroids, each of which has bounded canopy over $\matset$, and let $N \in \matset$ be such that
  \begin{enumerate}
    \item $N$ is 3-connected and not binary;
    \item $N$ stabilizes $\matset$ over $\pf$;
    \item all 3-connected $\pf$-representable matroids, which have an $N$-minor and are stabilized by $N$, are in $\matset$.
  \end{enumerate}
  Then there are finitely many excluded minors for $\matset$ having an $N$-minor.
\end{theorem}
\newcounter{rotasavecounter}
\setcounter{rotasavecounter}{\value{theorem}}

Of course the set $\matset$ we are most interested in is $\matset(\pf)$, but it might be possible to establish by other means that certain $\pf$-representable matroids do not occur as minors of some excluded minor. Then Theorem \ref{thm:rotapf} can be applied to a more restricted class.

The condition that the matroids in $\matset$ have bounded canopy is needed because our result depends crucially on Theorem~\ref{thm:rotabw}. At first it may seem like a rather strong restriction. However, it is expected that, if $\parf$ is a finitary partial field, \emph{every} matroid $N$ has bounded canopy over $\matset(\parf)$. The following is a weaker version of Conjecture 5.9 in \citet{GGW06b}.

\begin{conjecture}\label{con:boundcanopy}
  Let $N$ be a $\GF(q)$-representable matroid. There is an integer $l$, depending only on $N$ and $q$, such that, if $M$ is a $\GF(q)$-representable matroid with $\bw(M) > l$ and $N$ is a minor of $M$, then there exists an $e \in E(M)$ for which both $M\delete e$ and $M\contract e$ have a minor isomorphic to $N$.
\end{conjecture}

The difference with Geelen et al.'s conjecture is that they require that both $M\delete e$ and $M\contract e$ have a \emph{fixed} $N$-minor. Our conjecture is clearly implied by theirs. 

Our main application of Theorem \ref{thm:rotapf} is the following result:

\begin{theorem}\label{thm:gf5bcc}
	Rota's Conjecture for $\GF(5)$ is implied by Conjecture \ref{con:boundcanopy}.
\end{theorem}

Unfortunately we cannot make a similar statement for bigger finite fields, since our proof relies on the fact that 3-connected quinary matroids have a bounded number of inequivalent representations, a property that is not shared by bigger fields \cite{OVW95}.

Theorem \ref{thm:rotapf} comes very close to the following conjecture:

\begin{conjecture}
  Let $\parf$ be a partial field. If $\matset(\parf)$ has infinitely many excluded minors, then there is an infinite chain of matroids $N_1, N_2, \ldots$ such that $N_i$ has at least $i$ inequivalent representations over $\parf$, and such that $N_i$ is a minor of some excluded minor.
\end{conjecture}

The catch is in the observation that a matroid may not be stabilized by $N$ yet have fewer representations than $N$. We can, however, deduce the following:

\begin{corollary}\label{cor:infichain}
	Let $\parf$ be a partial field. If $\matset(\pf)$ has infinitely many excluded minors, but Conjecture \ref{con:boundcanopy} holds for $\pf$, then there is an infinite chain $N_1, N_2, \ldots$, with $N_i$ a minor of $N_{i+1}$ and $N_{i+1}$ not stabilized by $N_i$.
\end{corollary}

The paper is built up as follows. First, in Section \ref{sec:partialfields}, we give an overview of the theory of matroid representation over partial fields. Next, in Section \ref{sec:conn} we recall some standard results on connectivity. Section \ref{ssec:2seps} contains a few new results on $2$-separations. Section \ref{sec:fragility} contains a number of observations concerning fragility. In Section \ref{sec:incriminate} we use \emph{deletion pairs} to create a matrix over a partial field $\pf$ that should represent a matroid $M$ having an $N$-minor, if $M$ were representable over $\pf$. We introduce an \emph{incriminating set} which indicates where this particular representation fails. Deletion pairs and incriminating sets dictate the basic structure of the proof, in Section \ref{sec:thestrongproof}, of a weaker version of Theorem \ref{thm:rotapf}, in which $N$ is required to be a \emph{strong} stabilizer. In Section \ref{sec:theproof}, then, we show how to prove Theorem \ref{thm:rotapf} from this weaker version, and prove Corollary \ref{cor:infichain}. We conclude in Section \ref{sec:examples} with a number of applications of our result.

Unexplained notation follows \citet{oxley}. We write $\si(M)$ for the simplification of $M$ and $\co(M)$ for the cosimplification of $M$. We write $N\minorof M$ if $N$ is isomorphic to a minor of $M$. The smallest member of $\N$ is 0.

\section{Partial fields and representations}\label{sec:partialfields}

We start with the definition of a partial field. In this section we omit proofs, all of which can be found in at least one of \citep{SW96,PZ08conf,PZ08lift}. All proofs are also collected in Van Zwam \cite{vZ09}.

\begin{definition}\label{def:pf}
  A \emph{partial field} is a pair $(\ring, \group)$, where $\ring$ is a commutative ring and $\group$ is a subgroup of the group of units of $\ring$ such that $-1 \in \group$.
\end{definition}

In some contexts (for instance in Definition \ref{def:hom}) we may implicitly identify $\pf$ with the set $G \cup \{0\}$. Likewise, we say that $p$ is an \emph{element of} $\parf$ (notation: $p \in \parf$) if $p = 0$ or $p \in \group$. We define $\parf^* := \group$. Clearly, if $p,q \in \parf$ then also $p\cdot q \in \parf$, but $p+q$ need not be an element of $\parf$.

\begin{definition}\label{def:hom}
  Let $\parf_1,\parf_2$ be partial fields. A function $\phi:\parf_1\rightarrow\parf_2$ is a \emph{partial-field homomorphism} if
  \begin{enumerate}
    \item\label{it:hom1} $\phi(1) = 1$;
    \item\label{it:hom2} For all $p,q \in \parf_1$, $\phi(pq) = \phi(p)\phi(q)$;
    \item\label{it:hom3} For all $p,q,r \in \parf_1$ such that $p+q = r$, $\phi(p) + \phi(q) = \phi(r)$.
  \end{enumerate}
\end{definition}

Recall that $\pf$ is \emph{finitary} if there is a partial-field homomorphism $\pf\rightarrow\GF(q)$ for some prime power $q$. We single out some special homomorphisms:

\begin{definition}\label{def:pfisom}
  Let $\parf_1,\parf_2$ be partial fields and let $\phi:\parf_1\rightarrow\parf_2$ be a homomorphism. Then $\phi$ is an \emph{isomorphism} if\index{isomorphism!between partial fields|see{partial field}}
  \begin{enumerate}
    \item $\phi$ is a bijection;
    \item $\phi(p)+\phi(q) \in \parf_2$ if and only if $p+q \in \parf_1$.
  \end{enumerate}
\end{definition}

\begin{definition}
  A partial-field \emph{automorphism} is an isomorphism $\phi:\parf\rightarrow\parf$.\index{partial field!automorphism|defi}
\end{definition}

We introduce some notation related to matrices. Recall that formally, for linearly ordered sets $X$ and $Y$, an $X\times Y$ matrix $A$ over a partial field $\pf$ is a function $A:X\times Y \rightarrow \pf$. If $X=(1,2,\ldots,k)$ then we say that $A$ is a $k\times Y$ matrix.

If $X'\subseteq X$ and $Y'\subseteq Y$, then we denote by $A[X',Y']$ the submatrix of $A$ obtained by deleting all rows and columns in $X\setminus X'$, $Y\setminus Y'$. If $Z$ is a subset of $X\cup Y$ then we define $A[Z] := A[X\cap Z, Y \cap Z]$. Also, $A-Z := A[X\setminus Z, Y\setminus Z]$.

Let $A_1$ be an $X\times Y_1$ matrix over a partial field $\pf$ and $A_2$ an $X\times Y_2$ matrix over $\pf$, where $Y_1\cap Y_2 = \emptyset$. Then $A := [A_1 \  A_2]$ denotes the $X\times (Y_1 \cup Y_2)$ matrix with $A_{xy} = (A_1)_{xy}$ for $y \in Y_1$ and $A_{xy} = (A_2)_{xy}$ for $y \in Y_2$. If $X$ is an ordered set, then $I_X$ is the $X\times X$ identity matrix. If $A$ is an $X\times Y$ matrix over $\field$, then we use the shorthand $[I\  A]$ for $[I_X \  A]$.

Note that, for our purposes, the ordering of $X$ and $Y$ is only significant for the sign of determinants. And since the sign is irrelevant to the underlying matroid structure, we will freely permute rows and columns, always along with their labels, throughout the paper.

\begin{definition}\label{def:weakPmatrix}
  Let $\parf = (\ring, \group)$ be a partial field and let $A$ be a matrix with entries in $\ring$. Then $A$ is a \emph{$\parf$-matrix} if, for each square submatrix $D$ of $A$, $\det(D) \in \parf$.
\end{definition}

In particular, all entries of $A$ are in $\parf$. 

\begin{proposition}\label{prop:parfmatroid}
  Let $\parf = (\ring, \group)$ be a partial field, let $A$ be an $r\times E$ $\parf$-matrix, and define
  \begin{align*}
    \mathcal{B} := \big\{\,X \subseteq E \,:\, |X| = r, \det(A[r,X]) \neq 0\,\big\}.
  \end{align*}
  If $\mathcal{B} \neq \emptyset$ then $\mathcal{B}$ is the set of bases of a matroid.
\end{proposition}

Following the notation for matroids representable over fields, we denote the matroid of Proposition~\ref{prop:parfmatroid} by $M[A]$. Some more terminology:

\begin{definition}\label{def:representable}
  Let $M$ be a matroid. We say $M$ is \emph{representable} over a partial field $\parf$ (or, shorter, $\parf$-representable) if there exists a $\parf$-matrix $A$ such that $M = M[A]$. Moreover, we refer to $A$ as a \emph{representation matrix} of $M$ and say $M$ is \emph{represented by} $A$.
\end{definition}

\begin{proposition}\label{prop:pmatops}Let $A$ be a $\parf$-matrix. Then $A^T$ and $[I\  A]$ are also $\parf$-matrices. Let $\phi:\pf\rightarrow\pf'$ be a partial-field homomorphism. Then $\phi(A)$ is a $\pf'$-matrix and $M[I\  A] = M[I \  \phi(A)]$.
\end{proposition}

We will sometimes refer to the rank of a $\parf$-matrix.

\begin{definition}
  Let $A$ be an $X\times Y$ $\parf$-matrix. The \emph{rank} of $A$ is
  \begin{align*}
    \rank(A) := \max \big\{\, k\in \N \, :\, & \textrm{ there are } X'\subseteq X, Y'\subseteq Y \textrm{ with } |X'| = |Y'| = k, \\ & \textrm{ and } \det(A[X',Y']) \neq 0 \,\big\}.
  \end{align*}
\end{definition}

It is not hard to verify that the rank function is preserved by partial-field homomorphisms, and that it corresponds to the usual rank function if $\pf$ is a field.

\begin{definition}\label{def:pivot}Let $A$ be an $X\times Y$ matrix over a ring $\ring$ and let $x \in X, y \in Y$ be such that $A_{xy} \in \ring^*$. Then we define $A^{xy}$ to be the $(X\setminus x)\cup y \times (Y\setminus y)\cup x$ matrix with entries
\begin{align*}
  (A^{xy})_{uv} = \left\{ \begin{array}{ll}
    (A_{xy})^{-1} \quad & \textrm{if } (u,v) = (y,x)\\
    (A_{xy})^{-1} A_{xv} & \textrm{if } u = y, v\neq x\\
    -A_{uy} (A_{xy})^{-1} & \textrm{if } v = x, u \neq y\\
    A_{uv} - A_{uy} (A_{xy})^{-1} A_{xv} & \textrm{otherwise.}
  \end{array}\right.
\end{align*}
\end{definition}

We say that $A^{xy}$ is obtained from $A$ by \emph{pivoting} over $xy$. To give some intuition for this definition, we remark that it corresponds to row reduction in the matrix $[I_X\ A]$, as follows. Multiply row $x$ with $(A_{xy})^{-1}$, then add multiples of row $x$ to the other rows so the other entries in column $y$ become zero. Finally, exchange columns $x$ and $y$, and relabel row $x$ to $y$. The resulting matrix is $[I_{(X\setminus x)\cup y} \ A^{xy}]$. The next lemma formalizes this.

\begin{lemma}\label{lem:pivotmat}
  Let $A, x, y$ be as in Definition \ref{def:pivot}. Define $a := A_{xy}$, $b := A[X\setminus x,y]$, $X' := X\setminus x$, and
  \begin{align}
    F := \kbordermatrix{ & x & X'\\
                   y  & a^{-1} & 0 \cdots 0\\
                                                     & \phantom{0} &\\
                                                  X' & -a^{-1} b &    I_{X'}\\
                                                     & \phantom{0} &
                                                  }.\label{eq:Finverse}
  \end{align}
  Let $P$ be the $(X\cup Y) \times (X\cup Y)$ permutation matrix swapping $x$ and $y$. Then
  \begin{align*}
    F [I\  A] P = [I\  A^{xy}].
  \end{align*}
\end{lemma}

Note that $F$ is the inverse of $[I\  A][X,(y\cup X)\setminus x]$.

\begin{proposition}\label{prop:pivotproper}
  Let $A$ be an $X\times Y$  $\parf$-matrix and let $x \in X, y \in Y$ be such that $A_{xy} \neq 0$. Then $A^{xy}$ is a $\parf$-matrix.
\end{proposition}

We introduce some notions of equivalence of $\pf$-matrices.

\begin{definition}
  Let $A$, $A'$ be matrices with entries in a partial field $\pf$.
  \begin{enumerate}
    \item If $A'$ can be obtained from $A$ by repeatedly scaling rows and columns by elements of $\pf$, then we say that $A$ and $A'$ are \emph{scaling-equivalent}.
    \item If $A'$ can be obtained from $A$ by repeatedly scaling rows or columns, permuting rows, permuting columns, or pivoting, then we say that $A$ and $A'$ are \emph{geometrically equivalent}.
    \item If $\phi(A')$ is geometrically equivalent to $A$ for some partial-field automorphism $\phi$, then we say that $A'$ and $A$ are \emph{algebraically equivalent}.
  \end{enumerate}
\end{definition}

Note that in all operations, labels are exchanged along with their rows and columns. It is easy to verify that the defined relations are indeed equivalence relations, and that equivalent matrices represent the same matroid, as follows.

\begin{lemma}
  Let $A$, $A'$ be algebraically equivalent $\pf$-matrices. Then $M[I\ A] = M[I\ A']$.
\end{lemma}

 From this definition it is clear that there is a choice in how to count representations of a matroid. When we say ``$M$ has $k$ inequivalent representations'', we mean that $M$ has $k$ \emph{algebraically} inequivalent representations. In contrast, for the definition of a stabilizer below we use geometric equivalence.

In the remainder of the section we introduce some tools to help us to recognize when matrices are equivalent.

\begin{definition}\label{def:matroidgraph}
	Let $M$ be a matroid, $B$ a basis of $M$, and $D := E(M)\setminus B$. Then $G(M,B)$ is the bipartite graph with vertices $B\cup D$ and edges $\{bd : B \symdiff \{b,d\} \textrm{ is a basis of } M\}$.
\end{definition}

The graph $G(M,B)$ is the \emph{$B$-fundamental cocircuit incidence graph} of $M$ with respect to $B$ (cf. \cite[page 194]{oxley}). It has the following properties:

\begin{lemma}\label{lem:bipconn}
  Let $M$ be a matroid and $B$ a basis of $M$.
  \begin{enumerate}
    \item\label{bip:conn} $M$ is connected if and only if $G(M,B)$ is connected.
    \item\label{bip:threeconn} If $M$ is 3-connected, then $G(M,B)$ is 2-connected.
  \end{enumerate}
\end{lemma}

\begin{definition}\label{def:graph}
	Let $A$ be an $X\times Y$ matrix. Then $\bip(A)$ is the bipartite graph with vertices $X\cup Y$ and edges $\{xy : A_{xy}\neq 0\}$.
\end{definition}

\begin{lemma}\label{lem:graphsequal}
	Let $A$ be an $X\times Y$ $\pf$-matrix and $M := M[I\  A]$. Then $G(M,X) = G(A)$.
\end{lemma}

The following is a straightforward generalization of a well-known result by \citet{BL76} to partial fields \citep[see also][Theorem 6.4.7]{oxley}.

\begin{lemma}\label{lem:bippropertiesunique}
  Let $A$, $A'$ be matrices with entries in a partial field $\pf$. If $A'$ is scaling-equivalent to $A$ and $A'_{e} = A_{e}$ for all edges $e$ of a maximal spanning forest of $\bip(A)$, then $A' = A$.
\end{lemma}

Some more terminology: if $A_e = 1$ for all edges $e$ of a maximal spanning forest $T$ of $\bip(A)$, then we say $A$ is $T$-\emph{normalized}.

Finally, if two matrices are geometrically equivalent and have the same row labels, they are scaling-equivalent:

\begin{proposition}\label{prop:scalepivotcommute}
  Let $A$, $A'$ be geometrically equivalent $X\times Y$ $\parf$-matrices, where $X$, $Y$ are disjoint sets. Then $A$ is scaling-equivalent to $A'$.
\end{proposition}

\begin{proof}
  Since $A$ is geometrically equivalent to $A'$, we have
  \begin{align}
    [I_X\  A'] = F[I_X\  A]D\label{eq:scalepivotcommuteeq}
  \end{align}
  for an invertible matrix $F$ and a diagonal $(X\cup Y)\times (X\cup Y)$ matrix $D$, by Lemma \ref{lem:pivotmat}. From \eqref{eq:scalepivotcommuteeq} we conclude that
  \begin{align*}
    I_X = F I_X D[X,X].
  \end{align*}
  This implies that $F$ is a diagonal matrix. But then $A$ is scaling-equivalent to $A'$, as desired.
\end{proof}

\subsection{Stabilizers} 
\label{sub:stabilizers}
We now give a more precise definition of stabilizers.
\begin{definition}\label{def:stab}
	Let $\pf$ be a partial field, $M$ a matroid, $X$ a basis of $M$, $Y := E(M)\setminus X$, $S\subseteq X$, $T\subseteq Y$, and $N := M\contract S \delete T$.
	If, for all $X\times Y$ $\pf$-matrices $A_1, A_2$ such that
	\begin{enumerate}
		\item $M = M[I\  A_1] = M[I\  A_2]$
		\item $A_1[X\setminus S,Y\setminus T]$ is scaling-equivalent to $A_2[X\setminus S,Y\setminus T]$,
	\end{enumerate}
	we have that $A_1$ is scaling-equivalent to $A_2$, then we say that $N$ \emph{stabilizes} $M$.
\end{definition}

\begin{definition}\label{def:strongstab}
	If $N$ stabilizes $M$ over $\pf$, and \emph{every} representation of $N$ extends to a representation of $M$, then we say $N$ \emph{strongly stabilizes} $M$ over $\pf$.
\end{definition}

If $N$ has a unique representation over $\pf$ and $N$ stabilizes $M$, then $N$ is necessarily a strong stabilizer. Strong stabilizers were introduced by \citet{GOVW98}.

We say that $N$ stabilizes a set of matroids $\matset$ over a partial field $\pf$ if, for each 3-connected $M \in \matset$, every minor $M'$ isomorphic to $N$ stabilizes $M$ over $\pf$. The following is easily verified:

\begin{lemma}\label{lem:sistab}
	Let $M$ and $N$ be $\pf$-representable matroids such that $N\minorof M$ and $N$ stabilizes $\si(M)$ over $\pf$. Then $N$ stabilizes $M$ over $\pf$.
\end{lemma}


\section{Connectivity and branch width}\label{sec:conn}

\subsection{The connectivity function}
Recall the standard definition of the connectivity function:

\begin{definition}
  Let $M$ be a matroid with ground set $E$. The connectivity function of $M$, $\lambda_M:2^E\rightarrow \N$ is defined by
  \begin{align*}
    \lambda_M(Z) := \rank_M(Z) + \rank_M(E\setminus Z) - \rank(M).
  \end{align*}
\end{definition}

As usual, a $k$-\emph{separation} of $M$ is a partition $(X,Y)$ of $E(M)$ with $|X|, |Y| \geq k$ and $\lambda_M(X) < k$. A matroid is $k$-\emph{connected} if it has no separations of order $k-1$ or less. 

We start with some elementary and well-known properties of the connectivity function.

\begin{lemma}
  The function $\lambda_M$ is self-dual, submodular, and monotone under taking minors.
\end{lemma}

For representable matroids, the following lemma reformulates the connectivity function in terms of the ranks of certain submatrices of $A$.

\begin{lemma}[\citet{TruI}]\label{lem:matrixconn}
  Suppose $A$ is an $(X_1\cup X_2)\times(Y_1\cup Y_2)$ $\parf$-matrix (where $X_1, X_2, Y_1, Y_2$ are pairwise disjoint). Then
  \begin{align*}
    \lambda_{M[I\  A]}(X_1\cup Y_1) = \rank(A[X_1,Y_2]) + \rank(A[X_2,Y_1]).
  \end{align*}
\end{lemma}

To keep track of the connectivity of minors of $M$ it is convenient to introduce some extra notation.

\begin{definition}\label{def:basisminor}
  Let $M$ be a matroid, $B$ a basis of $M$, and $Y = E(M)\setminus B$. If $Z \subseteq E(M)$ then $M_B[Z] := M\contract (B\setminus Z) \delete (Y\setminus Z)$ and $M_B - Z := M_B[E\setminus Z]$.
\end{definition}

The following is easily seen:

\begin{lemma}
  If $M = M[I\  A]$ for an $X\times Y$ $\parf$-matrix $A$, sets $X$ and $Y$ are disjoint, and $Z \subseteq X\cup Y$, then $M_X[Z] = M[I\  A[Z]]$.
\end{lemma}

To counter the stacking of subscripts we introduce alternative notation for the connectivity function. This definition generalizes Lemma~\ref{lem:matrixconn} to arbitrary matroids $M$ and to arbitrary minors of $M$. It is equivalent to the definition found in \citet{GGK}.

\begin{definition}\label{def:connectivityB}
  Let $M$ be a matroid and $B$ a basis of $M$. Then $\lambda_B:2^{E(M)}\times 2^{E(M)} \rightarrow \N$ is defined as
  \begin{align*}
    \lambda_B(X,Y) := \rank_{M\contract (B\setminus Y)}(X\setminus B) + \rank_{M\contract (B\setminus X)}(Y\setminus B)
  \end{align*}
  for all $X,Y\subseteq E(M)$.
\end{definition}

The following lemma shows that this is indeed the connectivity function of a minor of $M$ when $X$ and $Y$ are disjoint. Once again we omit the straightforward proof.

\begin{lemma}\label{lem:connectivityoldnew}
  Let $M$ be a matroid, $B$ a basis of $M$, and $X,Y$ disjoint subsets of $E(M)$. Then
  \begin{align*}
  \lambda_{B}(X,Y) = \lambda_{M_B[X\cup Y]}(X).
  \end{align*}
\end{lemma}

The following two results can be found in \citet[{Proposition 4.3.6, Corollary 11.2.1}]{oxley}.

\begin{theorem}\label{thm:2connchain}
  Let $M$ and $N$ be connected matroids, $N\minorof M$, with $|E(N)|<|E(M)|$. Then there is an $e \in E(M)$ such that some $M' \in \{M\delete e, M\contract e\}$ is connected with $N\minorof M'$.
\end{theorem}

\begin{theorem}[Splitter Theorem]\label{thm:3connchain}
  Let $M$ and $N$ be $3$-connected matroids, $N\minorof M$, with $|E(M)| > |E(N)| \geq 4$, such that $M$ is not isomorphic to a wheel or a whirl. Then there is an $e \in E(M)$ such that some $M' \in \{M\delete e, M\contract e\}$ is $3$-connected with $N\minorof M'$.
\end{theorem}

\subsection{Blocking sequences}

The following definitions are from \citet{GGK}.

\begin{definition}
  Let $M$ be a matroid on ground set $E$, $M'$ a minor of $M$ on ground set $E'\subseteq E$, and $(Z_1',Z_2')$ a $k$-separation of $M'$. We say that $(Z_1', Z_2')$ is \emph{induced} in $M$ if there exists a $k$-separation $(Z_1,Z_2)$ of $M$ with $Z_1' \subseteq Z_1$ and $Z_2' \subseteq Z_2$.
\end{definition}

Let $B$ be a basis of $M$ such that $M' = M_B[E']$.

\begin{definition}\label{def:blockseq}
	Let $M$, $M'$, $E$, $E'$, $Z_1'$, and $Z_2'$ be as in the previous definition. A \emph{blocking sequence}\index{blocking sequence} for $(Z_1',Z_2')$ is a sequence of elements $v_1, \ldots, v_t$ of $E\setminus E'$ such that
  \begin{enumerate}
    \item\label{bls:first} $\lambda_B(Z_1', Z_2' \cup v_1) = k$;
    \item\label{bls:middle} $\lambda_{B}(Z_1'\cup v_i, Z_2'\cup v_{i+1}) = k$ for $i = 1, \ldots, t-1$;
    \item\label{bls:last} $\lambda_{B}(Z_1'\cup v_t,Z_2') = k$; and
    \item no proper subsequence of $v_1, \ldots, v_t$ satisfies the first three properties.
  \end{enumerate}
\end{definition}

Blocking sequences find their origin in Seymour's work on regular matroid decomposition \citep[Section 8]{Sey80}. The first general formulation was due to \citet{TruIII}, but blocking sequences truly took off with the publication of the proof of Rota's Conjecture for $\GF(4)$ \citep{GGK}. We have opted to use their notation rather than the notation used in, for instance, \citet{GHW05}, because Definition~\ref{def:blockseq} clearly exhibits the symmetry.

The following theorem illustrates the usefulness of blocking sequences:

\begin{theorem}[\citet{GGK}, Theorem 4.14]\label{thm:blockseq}
  Let $M$ be a matroid on ground set $E$, $B$ a basis of $M$, $M' := M_B[E']$ for some $E'\subseteq E$, and $(Z_1',Z_2')$ an exact $k$-separation of $M'$. Exactly one of the following holds:\index{k-separation@$k$-separation}
  \begin{enumerate}
    \item There exists a blocking sequence for $(Z_1',Z_2')$;
    \item $(Z_1',Z_2')$ is induced in $M$.
  \end{enumerate}
\end{theorem}

In the first case we say that $(Z_1',Z_2')$ is \emph{bridged} in $M$.

Another useful property of blocking sequences is the following:

\begin{lemma}[\citet{GGK}, Proposition 4.15(iv)]\label{lem:blockseqalternating}If $v_1, \ldots, v_t$ is a blocking sequence for the $k$-separation $(Z_1',Z_2')$, then $v_i\in B$ implies $v_{i+1} \in E\setminus B$ and $v_i \in E\setminus B$ implies $v_{i+1} \in B$ for $i = 1, \ldots, t-1$.
\end{lemma}

We will use the following lemma:

\begin{lemma}[\citet{GGK}, Proposition 4.16(i)]\label{lem:blseqlem}
  Let $v_1, \ldots, v_t$ be a blocking sequence for $(Z_1', Z_2')$. If $Z_2'' \subseteq Z_2'$ is such that $|Z_2''| \geq k$ and $\lambda_B(Z_1',Z_2'') = k-1$, then $v_1, \ldots, v_{t-1}$ is a blocking sequence for the exact $k$-separation $(Z_1',Z_2''\cup v_t)$.
\end{lemma}

\subsection{Branch width}

A graph $T = (V,E)$ is a \emph{cubic tree} if $T$ is a tree in which each vertex has degree exactly one or three. We denote the leaves of $T$ by $L(T)$.

\begin{definition}
  Let $M$ be a matroid. A \emph{partial branch decomposition} of $M$ is a pair $(T, l)$, where $T$ is a cubic tree and $l:V(T)\rightarrow 2^{E(M)}$ a function assigning a subset of $E(M)$ to each vertex of $T$, such that $\{l(v)  :  v \in V(T)\}$ partitions $E(M)$.
\end{definition}

If $T$ is a tree and $e=vw \in E(T)$, then we denote by $T_v$ the component of $T\delete e$ containing $v$.

\begin{definition}
  Let $M$ be a matroid and let $(T,l)$ be a partial branch decomposition of $M$. We define $w_{(T,l)}: V^2\rightarrow \N$ as
  \begin{align*}
    w_{(T,l)}(v,w) & = \left\{ \begin{array}{cl}
                       \lambda_M\left(\bigcup_{u \in V(T_v)} l(u)\right) + 1 & \textrm{ if } vw \in E(T);\\
                       \phantom{\bigg(} 0 & \textrm{ otherwise.}
                       \end{array}\right.
  \end{align*}
\end{definition}

In words, $w_{(T,l)}(v,w)$ is the degree of the separation of $M$ displayed by the edge $vw$. Note that $(\bigcup_{u \in V(T_v)} l(u), \bigcup_{u \in V(T_w)} l(u))$ is a partition of $E(M)$, so $w_{(T,l)}(v,w) = w_{(T,l)}(w,v)$. Hence, for $e = vw \in E(T)$, we will write $w_{(T,l)}(e)$ as shorthand for $w_{(T,l)}(v,w)$.

\begin{definition}
  Let $M$ be a matroid and let $(T,l)$ be a partial branch decomposition of $M$. The \emph{width} of $(T,l)$ is
  \begin{align*}
    w(T,l) := \left\{ \begin{array}{ll}
                      \max_{e \in E(T)} w_{(T,l)}(e) & \textrm{ if } E(T) \neq \emptyset\\
                      1                      & \textrm{ otherwise.}
                      \end{array}\right.
  \end{align*}
\end{definition}

\begin{definition}
  Let $M$ be a matroid. A \emph{branch decomposition} of $M$ is a partial branch decomposition such that $|l(v)| \leq 1$ for all $v \in L(T)$, and such that $l(v) = \emptyset$ for all $v \in V(T)\setminus L(T)$.
\end{definition}

\begin{definition}
  Let $M$ be a matroid. A \emph{reduced branch decomposition} of $M$ is a branch decomposition such that $|l(v)| = 1$ for all $v \in L(T)$.
\end{definition}

We denote the set of reduced branch decompositions of $M$ by $\mathcal{D}_M$.

\begin{definition}
  Let $M$ be a matroid. The \emph{branch width} of $M$ is
  \begin{align*}
    \bw(M) := \min_{(T,l) \in \mathcal{D}_M} w(T,l).
  \end{align*}
\end{definition}

We start with some elementary and well-known observations. We omit the proofs.

\begin{lemma}
  Let $(T,l)$ be a branch decomposition of a matroid $M$. There is a reduced branch decomposition $(T',l')$ of $M$ such that $w(T,l) = w(T',l')$.
\end{lemma}

\begin{proposition}\label{prop:bwmonotone}
  Let $M$ be a matroid and $e \in E(M)$. Then
  \begin{align*}
    \bw(M\delete e) \leq \bw(M) \leq \bw(M\delete e) + 1.
  \end{align*}
\end{proposition}

Series and parallel classes do not have an effect on the branch width of a matroid:

\begin{proposition}Let $M$ be a matroid with $\bw(M)\geq 2$. Then $\bw(M) = \bw(\si(M))$.
\end{proposition}

\citet[Theorem 1.4]{GHW05} proved the following result, which states that a blocking sequence does not increase branch width by much:

\begin{theorem}\label{thm:blseqbranchwidth}
  Let $M$ be a matroid having basis $B$, and let $Z \subseteq E(M)$. Suppose that $M_B[Z]$ has a $k$-separation $(X,Y)$, and that $v_1, \ldots, v_t$ is a blocking sequence for $(X,Y)$ in $M$. Then $\bw(M_B[Z\cup\{v_1, \ldots, v_t\}]) \leq \bw(M_B[Z]) + k$.\index{blocking sequence}
\end{theorem}

We note one particular case for the examples in Section \ref{sec:examples}:

\begin{lemma}\label{lem:bwwhirl}
  For all $n \geq 2$, $\bw(\whirl{n}) = 3$.\index{whirl}
\end{lemma}

\subsection{Results on $2$-separations}\label{ssec:2seps}

We will need to bound the number of $2$-separations in small extensions of a $3$-connected matroid. The following lemma does just that.

\begin{lemma}\label{lem:2sepbound}
  If $M$ is a connected matroid, $N \minorof M$, $N$ is $3$-connected, $|E(N)| \geq 4$, and $|E(M)|-|E(N)| \leq k$, then the number of 2-separations in $M$ is at most $2^{k+1}$.
\end{lemma}

\begin{proof}
  Let $t_k$ denote the maximum number of $2$-separations of a $k$-element extension of a $3$-connected matroid. We argue by induction on $k$. By Theorem~\ref{thm:2connchain} there exist a basis $B$ of $M$, a subset $X$ of $E(M)$, and an ordering $e_1, \ldots, e_k$ of the elements of $E(M)\setminus X$ such that $N \cong M_B[X]$ and $M_B[X\cup \{e_1, \ldots, e_i\}]$ is connected for all $i \in \{1, \ldots, k\}$.

  If $k = 1$ then $e_1$ can be in series or in parallel with at most one element of $M_B[X]$, and it cannot be both in series and in parallel. Hence $t_1 = 1$.

  By duality we may assume $e_k \not \in B$. Let $(Z_1,Z_2)$ be a $2$-separation of $M$, with $e_k \in Z_1$. If $|Z_1| \geq 3$ then $\lambda_{M\delete e_k}(Z_2) \leq 1$, and connectivity of $M\delete e_k$ implies that equality holds. Hence $(Z_1\setminus e_k, Z_2)$ is a $2$-separation of $M\delete e_k$. This leads to at most two $2$-separations of $M$: $(Z_1,Z_2)$ and $(Z_1\setminus e_k, Z_2\cup e_k)$.

  If a $2$-separation of $M$ is not an extension of a $2$-separation of $M\delete e_k$, then we must have $|Z_1| = 2$. There is one of these for each $f \in E(M)\setminus \{e_k\}$ such that $e_k, f$ are in series or in parallel. But $e_k$ can, again, be in series or in parallel with at most one element of $X$, as well as with each of $e_1, \ldots, e_{k-1}$, so it follows that
  \begin{align*}
    t_k \leq 2 t_{k-1} + k.
  \end{align*}
  Define $t'_k := 2^{k+1} - k - 2$. We claim that $t_k \leq t_k'$. Indeed: $t'_1 = t_1 = 1$, and if the claim is valid for $k-1$, then
  \begin{align*}
    t_k \leq 2 t_{k-1} + k \leq 2 t'_{k-1} + k = 2(2^k- (k-1)-2)+k = 2^{k+1} -k - 2 = t'_k.
  \end{align*}
  Obviously $t'_k \leq 2^{k+1}$, and the result follows.
\end{proof}

The following definitions are from \citet{GGK}.

\begin{definition}
  Let $M$ be a matroid and let $(X_1,X_2)$ and $(Y_1,Y_2)$ be $2$-separations of $M$. If $X_i \cap Y_j\neq \emptyset$ for all $i,j \in \{1,2\}$, then we say that $(X_1,Y_1)$ and $(X_2,Y_2)$ \emph{cross}.
\end{definition}

\begin{definition}
  Let $M$ be a matroid and let $(X_1,X_2)$ be a $2$-separation of $M$. We say that $(X_1,X_2)$ is \emph{crossed} if there exists a $2$-separation $(Y_1,Y_2)$ of $M$ such that $(X_1,X_2)$ and $(Y_1,Y_2)$ cross. Otherwise we say $(X_1,X_2)$ is \emph{uncrossed}.
\end{definition}

Crossing $2$-separations have previously been studied by \citet{CE80}. \citet{OSW04} characterized crossing $3$-separations in $3$-connected matroids, and those results have been generalized to crossing $k$-separations by \citet{AO08}. The proof of the following lemma is an instance of the technique of ``uncrossing'' from those papers.

\begin{lemma}\label{lem:uncrossed2sep}
	Let $M$ be a connected, nonbinary matroid. If $M$ has a 2-separation, then $M$ must have an uncrossed 2-separation.
\end{lemma}

\begin{proof}
  Since $M$ is non-binary, $M$ has a $U_{2,4}$-minor. Fix such a minor, say with elements $\{a,b,c,d\}$. If $(X,Y)$ is a $2$-separation of $M$, then either $|X\cap\{a,b,c,d\}| \leq 1$ or $|Y \cap \{a,b,c,d\}| \leq 1$. Let $(X',Y')$ be a $2$-separation of $M$ such that $Y'$ is maximal subject to $|Y' \cap \{a,b,c,d\}| \leq 1$. Let $(U,V)$ be a $2$-separation that crosses $(X',Y')$, and assume $|V\cap\{a,b,c,d\}|\leq 1$. Then $X'\cap U$ has at least two elements from $\{a,b,c,d\}$. Now
  \begin{align*}
    2 = \lambda_M(X') + \lambda_M(U) \geq \lambda_M(X'\cap U) + \lambda_M(X'\cup U),
  \end{align*}
  so we must have $\lambda_M(X'\cap U) = 1 = \lambda_M(Y' \cup V)$. Since $|(X'\cap U)\cap \{a,b,c,d\}| \geq 2$, it follows that $|(Y'\cup V)\cap \{a,b,c,d\}| \leq 1$. But $|Y'\cup V| > |Y'|$, a contradiction.
\end{proof}

Uncrossed $2$-separations are relevant because they can be bridged without introducing new $2$-separations:

\begin{lemma}[\citet{GGK}, Proposition 4.17]\label{lem:block2sep}
  Let $M$ be a matroid, $B$ a basis of $M$, $E' \subseteq E$, and $(Z_1',Z_2')$ an uncrossed $2$-separation of \index{2-separation@$2$-separation} $M_B[E']$. Let $v_1, \ldots, v_t$ be a blocking sequence for $(Z_1',Z_2')$. If $(Z_1,Z_2)$ is a $2$-separation of $M_B[E'\cup \{v_1, \ldots, v_t\}]$ then $Z_i' \cup \{v_1, \ldots, v_t\} \subseteq Z_j$  for some $i, j \in \{1,2\}$.
\end{lemma}

\begin{corollary}\label{cor:block2sep}
  Let $M$ be a matroid, $B$ a basis of $M$, $E' \subseteq E$, and $(Z_1',Z_2')$ an uncrossed $2$-separation of the connected matroid $M_B[E']$. Let $v_1, \ldots, v_t$ be a blocking sequence for $(Z_1',Z_2')$. Then $M_B[E'\cup \{v_1, \ldots, v_t\}]$ has strictly fewer $2$-separations than $M_B[E']$.
\end{corollary}

\begin{proof}
  Let $(Z_1,Z_2)$ be a $2$-separation of $M_B[E'\cup\{v_1,\ldots,v_t\}]$. Possibly after relabelling, Lemma~\ref{lem:block2sep} implies that $Z_2'\cup \{v_1,\ldots,v_t\} \subseteq Z_2$. Therefore we know that $|Z_2\setminus\{v_1,\ldots,v_t\}| \geq 2$. Also $|Z_1| \geq 2$ so, since $M_B[E']$ is connected, $1 \leq \lambda_B(Z_1,Z_2\setminus \{v_1, \ldots, v_t\}) \leq \lambda_B(Z_1,Z_2)  = 1$. Hence $(Z_1, Z_2\setminus\{v_1,\ldots,v_t\})$ is a $2$-separation of $M_B[E']$, and the result follows.
\end{proof}

\subsection{Excluded minors for well-closed classes}

We omit the easy proofs of the observations in this section. In all results, $\matset$ is a well-closed class of matroids.

\begin{lemma}\label{lem:exmindual}
  Let $M$ be an excluded minor for $\matset$. Then $M^*$ is an excluded minor for $\matset$.
\end{lemma}

\begin{lemma}\label{lem:exmin3c}
  Let $M$ be an excluded minor for $\matset$. Then $M$ is $3$-connected.
\end{lemma}

\begin{lemma}\label{lem:exminrankbound}
  Suppose all matroids in $\matset$ are representable over some finite field $\GF(q)$. Let $r \in \N$. Then there are finitely many rank-$r$ excluded minors for $\matset$.
\end{lemma}

\section{Fragility}\label{sec:fragility}

In the introduction we defined fragility for a single matroid. A slightly more general definition is the following:

\begin{definition}\label{def:fragile}
	Let $\mathcal{N}$ be a set of matroids. A matroid $M$ is $\mathcal{N}$-\emph{fragile} if, for all $e \in E(M)$, at least one of $M\delete e$ and $M\contract e$ has no minor isomorphic to a member of $\mathcal{N}$. Moreover, an $\mathcal{N}$-fragile matroid $M$ is \emph{strictly} $\mathcal{N}$-fragile if some minor of $M$ is isomorphic to a member of $\mathcal{N}$.
\end{definition}

Let $N$ be a matroid. We say that a matroid $M$ is $N$-\emph{fragile} if $M$ is $\{N\}$-fragile. We establish a few basic properties of $\mathcal{N}$-fragile matroids. The following is easy to see from the definition:

\begin{lemma}
  If $M$ is $\mathcal{N}$-fragile and $M'\minorof M$ then $M'$ is $\mathcal{N}$-fragile.
\end{lemma}

The following proposition is well-known; see, for instance, \citet[Corollary 2.4]{GW01} for a proof technique.

\begin{proposition}\label{pro:twosepminorintersection}
	Let $M$ be a matroid with a 2-separation $(A,B)$ and let $N$ be a 3-connected minor of $M$. Assume $|E(N)\cap A| \geq |E(N)\cap B|$. Then $|E(N)\cap B| \leq 1$. Moreover, unless $B$ consists of a parallel class or series class, there is an $e\in B$ such that both $M\delete e$ and $M\contract e$ have a minor isomorphic to $N$.
\end{proposition}

An immediate corollary is the following. It was also proven by Kingan and Lemos \cite[Proposition 3.1]{KL02}.

\begin{proposition}\label{prop:almostconn}
  Let $\mathcal{N}$ be a set of 3-connected matroids with $|E(N)| \geq 4$ for all $N \in \mathcal{N}$, and let $M$ be a strictly $\mathcal{N}$-fragile matroid. Then $M$ is $3$-connected up to series and parallel classes.
\end{proposition}

Some more terminology:

\begin{definition}
  Let $\mathcal{N}$ be a set of matroids, let $M$ be a matroid, and let $e \in E(M)$.
  \begin{enumerate}
    \item If $M\contract e$ has a minor isomorphic to a member of $\mathcal{N}$ then $e$ is \emph{$\mathcal{N}$-contractible};
    \item If $M\delete e$ has a minor isomorphic to a member of $\mathcal{N}$ then $e$ is \emph{$\mathcal{N}$-deletable};
    \item If neither $M\delete e$ nor $M\contract e$ has a minor isomorphic to a member of $\mathcal{N}$ then $e$ is \emph{$\mathcal{N}$-essential}.
  \end{enumerate}
\end{definition}

We will drop the prefix ``$\mathcal{N}$-'' if it is clear from the context which set is intended. For readers familiar with the work of \citet{TruVI} this definition may cause some confusion: Truemper defines a \emph{con} element $e$ to be such that $M\contract e$ has no $F_7$-minor and no $F_7^*$-minor, and a \emph{del} element $e$ to be such that $M\delete e$ has no $F_7$- and no $F_7^*$-minor. The reasoning behind his choice is clear: rather than studying $\{F_7,F_7^*\}$-fragile binary matroids, he studies \emph{almost regular} binary matroids. Hence losing the minor is a good thing for him. For us the elements of $\mathcal{N}$ will be stabilizers, so we want to keep a member of $\mathcal{N}$ by all means. We use the following notation:

\begin{definition}
  Let $\mathcal{N}$ be a set of matroids and let $M$ be a matroid.
  \begin{align*}
    \conset_{\mathcal{N},M} & := \{\, e \in E(M)  :  e \textrm{ is } \mathcal{N}\textrm{-contractible}\,\};\\
    \delset_{\mathcal{N},M} & := \{\, e \in E(M)  :  e \textrm{ is } \mathcal{N}\textrm{-deletable}\,\};\\
    \essset_{\mathcal{N},M} & := \{\, e \in E(M)  :  e \textrm{ is } \mathcal{N}\textrm{-essential}\,\}.
  \end{align*}
\end{definition}

We conclude the section with a number of elementary properties of $\mathcal{N}$-fragile matroids. We omit the straightforward proofs.

\begin{lemma}\label{lem:almostpartition}
	Let $\mathcal{N}$ be a set of matroids, and let $M$ be an $\mathcal{N}$-fragile matroid.
	\begin{enumerate}
		\item $\conset_{\mathcal{N},M}$, $\delset_{\mathcal{N},M}$, $\essset_{\mathcal{N},M}$ are pairwise disjoint and partition $E(M)$.
		\item Let $\mathcal{N}^* := \{N^*  :  N \in \mathcal{N}\}$. Then $M^*$ is $\mathcal{N}^*$-fragile with $\conset_{\mathcal{N}^*, M^*} = \delset_{\mathcal{N},M}$, $\delset_{\mathcal{N}^*, M^*} = \conset_{\mathcal{N},M}$, and $\essset_{\mathcal{N}^*, M^*} = \essset_{\mathcal{N},M}$.
		\item Let $M'\minorof M$.
	  \begin{enumerate}
	    \item If $e \in E(M')$ and $e\in \conset_{\mathcal{N},M}$ then $e \in \conset_{\mathcal{N},M'}\cup \essset_{\mathcal{N},M'}$;
	    \item If $e \in E(M')$ and $e\in \delset_{\mathcal{N},M}$ then $e \in \delset_{\mathcal{N},M'}\cup \essset_{\mathcal{N},M'}$;
	    \item If $e \in E(M')$ and $e\in \essset_{\mathcal{N},M}$ then $e \in \essset_{\mathcal{N},M'}$.
	  \end{enumerate}
		\item\label{it:parpairdel} If $N$ is 3-connected and $|E(N)|\geq 4$ for all $N\in\mathcal{N}$, and if $\rank_M(\{e,f\}) = 1$, then $e$ and $f$ are both deletable.
		\item If $N$ is 3-connected and $|E(N)|\geq 4$ for all $N\in\mathcal{N}$, and if $\rank^*_M(\{e,f\}) = 1$, then $e$ and $f$ are both contractible.
	\end{enumerate}
\end{lemma}

\section{Deletion pairs and incriminating sets}\label{sec:incriminate}

The results in this section form part of the basic strategy of our proof. They are closely related to results in \citet{GGK} and \citet{HMZ11}. Our first ingredient is an easy corollary of a theorem by \citet{Whi96b}. We start by defining a \emph{deletion pair}.

\begin{definition}\label{def:delpair}
  Let $M$ be a matroid having an $N$-minor. Then $\{u,v\} \subseteq E(M)$ is a \emph{deletion pair preserving} $N$ if $M\delete \{u,v\}$ is connected and $\co(M\delete u)$, $\co(M\delete v)$, $\co(M\delete \{u,v\})$ are $3$-connected and have an $N$-minor.
\end{definition}

A deletion pair is guaranteed to exist, provided that $M$ is sufficiently large and 3-connected:

\begin{theorem}[\citet{Whi96b}, Theorem 3.2]\label{thm:secondelt}
  Let $M$, $N$ be matroids such that $N\minorof M$, $\rank(M) -\rank(N) \geq 3$, and both $M$ and $N$ are $3$-connected. If there exists a $u \in E(M)$ such that $\si(M\contract u)$ is $3$-connected and has an $N$-minor, then there exists a $v \in E(M)$, $v\neq u$, such that $\si(M\contract v)$ and $\si(M\contract \{u,v\})$ are both $3$-connected, and $\si(M\contract\{u,v\})$ has an $N$-minor.
\end{theorem}

\begin{corollary}\label{cor:delpairexists}
  Let $M$ and $N$ be $3$-connected matroids, with $N\minorof M$, and suppose $M$ is not a wheel or a whirl. If $\rank(M) - \rank(N) \geq 3$ and $\rank(M^*) - \rank(N^*) \geq 3$, then for some $(M',N') \in \{(M,N), (M^*, N^*)\}$, $M'$ has a deletion pair $\{u,v\}$ preserving $N'$. Moreover, $\{u,v\}$ can be chosen such that $M'\delete u$ is $3$-connected.
\end{corollary}

\begin{proof}
  By the Splitter Theorem there is a $u \in E(M)$ such that either $M\delete u$ is $3$-connected with an $N$-minor, or $M\contract u$ is $3$-connected with an $N$-minor. Using duality we may assume, without loss of generality, that the former holds. Then the dual of Theorem~\ref{thm:secondelt} implies the existence of a $v \in E(M)\setminus u$ such that $\co(M\delete v)$ and $\co(M\delete \{u,v\})$ are $3$-connected with an $N$-minor. To ensure that $\{u,v\}$ is a deletion pair we need to prove that $M\delete \{u,v\}$ is connected. But $M\delete \{u,v\} = (M\delete u)\delete v$, and since $M\delete u$ is $3$-connected, $M\delete \{u,v\}$ is $2$-connected.
\end{proof}

In the remainder of this section $\parf$ will be a partial field, $\matset$ will be a well-closed class of $\pf$-representable matroids, $N \in \matset$ will be a $3$-connected $\parf$-representable matroid that is a strong $\parf$-stabilizer for $\matset$, $M$ will be a $3$-connected matroid with an $N$-minor, and $\{u,v\}\subseteq E(M)$ will be a deletion pair preserving $N$.

Next we employ the deletion pair to create a candidate $\parf$-representation for $M$ when $M\delete u$ and $M\delete v$ are $\pf$-representable.

\begin{lemma}\label{lem:delpairscale}
  Let $D$, $D'$ be $X\times Y$ matrices with entries in a partial field $\parf$. Let $u,v \in Y$ be such that
  \begin{enumerate}
    \item $D-u$ is scaling-equivalent to $D'-u$ and $D-v$ is scaling-equivalent to $D'-v$;
    \item $G(D-\{u,v\})$ is connected.
  \end{enumerate}
  Then $D$ is scaling-equivalent to $D'$.
\end{lemma}

\begin{proof}
  If one of $D[X,u]$ and $D[X,v]$ is an all-zero column then the result is trivially true, so we assume this is not the case. Now let $T'$ be a spanning tree for $\bip(D-\{u,v\})$ and let $T := T' \cup \{xu, x'v\}$ for some $x, x' \in X$ with $D_{xu} \neq 0$, $D_{x'v} \neq 0$. Then $T$ is a spanning tree for $\bip(D) = \bip(D')$. Assume, without loss of generality, that $D$ and $D'$ are $T$-normalized. Then $D-u$ and $D'-u$ are $(T\setminus xu)$-normalized, and hence, by Lemma~\ref{lem:bippropertiesunique}, $D-u = D'-u$. Likewise $D-v = D'-v$. But then $D = D'$, and the result follows.
\end{proof}

\begin{theorem}\label{thm:uniquematrix}
  Let $D$ be an $X_N\times Y_N$ $\parf$-matrix such that $N = M[I\  D]$. Choose sets $B, E_N \subseteq E(M)$ such that $B$ is a basis of $M\delete \{u,v\}$, $E_N\subseteq E(M)\setminus\{u,v\}$ is such that $M_B[E_N] = N$, and $X_N\subseteq B$. Suppose $M\delete u, M\delete v \in \matset$.
  Then there is a $B\times (E(M)\setminus B)$ matrix $A$ with entries in $\parf$ such that
  \begin{enumerate}
    \item $A-u$ and $A-v$ are $\parf$-matrices;
    \item $M[I\  (A-u)] = M\delete u$ and $M[I\  (A-v)] = M\delete v$;
    \item $A[E_N]$ is scaling-equivalent to $D$.
  \end{enumerate}
  Moreover, $A$ is unique up to scaling of rows and columns.
\end{theorem}

\begin{proof}
  Suppose $D$, $B$, $E_N$ are as in the theorem. Let $T$ be a spanning tree for $G(M,B)$ having $u$ and $v$ as leaves; $T$ exists since $\{u,v\}$ is a deletion pair. The fact that $N$ is a strong $\parf$-stabilizer for $\matset$, together with the dual of Lemma \ref{lem:sistab}, shows that there is a unique $(T\setminus u)$-normalized $\parf$-matrix $A'$ such that $A'[E_N]$ is scaling-equivalent to $D$ and $M\delete u = M[I\  A']$, and a unique $(T\setminus v)$-normalized $\parf$-matrix $A''$ such that $A''[E_N]$ is scaling-equivalent to $D$ and $M\delete v = M[I\  A'']$. Since $N$ is a strong $\parf$-stabilizer, also $A'-v = A''- u$. Now let $A$ be the matrix obtained from $A'$ by appending column $A''[B,v]$. Then $A$ satisfies all properties of the theorem. Uniqueness follows from Lemma~\ref{lem:delpairscale}.
\end{proof}

Most of the time we will apply Theorem~\ref{thm:uniquematrix} to matrices $D$ that do not extend to a representation of $M$. If a matrix with entries in a partial field does not represent a matroid, then it must have one of three problems, described by the next definition.

\begin{definition}\label{def:incrimi}
  Let $B$ be a basis of $M$ and let $A$ be a $B\times (E(M)\setminus B)$ matrix with entries in $\parf$. A set $Z\subseteq E(M)$ \emph{incriminates} the pair $(M,A)$ if $A[Z]$ is square and one of the following holds:\index{incriminating set|defi}
  \begin{enumerate}
    \item\label{it:incrimi1} $\det(A[Z])\not\in\parf$;
    \item\label{it:incrimi2} $\det(A[Z]) = 0$ but $B\symdiff Z$ is a basis of $M$;
    \item\label{it:incrimi3} $\det(A[Z]) \neq 0$ but $B\symdiff Z$ is dependent in $M$.
  \end{enumerate}
\end{definition} 

The proof of the following lemma is obvious and therefore omitted.

\begin{lemma}\label{lem:incriminate}
  Let $A$ be an $X\times Y$ matrix, where $X$ and $Y$ are disjoint and $X\cup Y = E(M)$. Exactly one of the following statements is true:
  \begin{enumerate}
    \item $A$ is a $\parf$-matrix and $M = M[I\  A]$;
    \item some $Z\subseteq X\cup Y$ incriminates $(M,A)$.
  \end{enumerate}
\end{lemma}

For the remainder of this section we will assume that $A$ is an $X\times Y$ matrix with entries in $\parf$ such that $X$ and $Y$ are disjoint, $X\cup Y = E(M)$, and $u,v\in Y$.

It is often desirable to have a small incriminating set. If we have some information about minors of $A$ then this can be achieved by pivoting.

\begin{theorem}\label{thm:incriminatingsetsmall}
  Suppose $A-u$, $A-v$ are $\parf$-matrices and $M\delete u = M[I\  (A-u)]$, $M\delete v = M[I\  (A-v)]$. Suppose $Z \subseteq X\cup Y$ incriminates $(M,A)$. Then there exists an $X'\times Y'$ matrix $A'$ and $a,b \in X'$, such that $u,v \in Y'$, $A-u$ is geometrically equivalent to $A'-u$, such that $A-v$ is geometrically equivalent to $A'-v$, and such that $\{a,b,u,v\}$ incriminates $(M,A')$.
\end{theorem}

\begin{proof}
  Suppose the theorem is false. Let $X,Y,A,u,v,M,Z$ form a counterexample, and suppose the counterexample was chosen such that $|Z\cap Y|$ is minimal. Clearly $u,v \in Z$. Suppose $y \in Z$ for some $y \in Y\setminus\{u,v\}$.
  \begin{claim}
    Some entry of $A[X\cap Z, y]$ is nonzero.
  \end{claim}
  \begin{subproof}
    Suppose all entries of $A[X\cap Z, y]$ equal zero. Then $\det(A[Z]) = 0$. Since $Z$ incriminates $(M,A)$, this implies that $X\symdiff Z$ is a basis of $M$. Now there is an $x \in Z\cap X$ such that $B := X\symdiff \{x,y\}$ is a basis of $M$. But since $u,v \not \in B$, $B$ is also a basis of $M\delete \{u,v\}$. Since $M\delete u = M[I\  (A-u)]$, this implies that $A_{xy} \neq 0$, a contradiction.
  \end{subproof}
  Now pick $x \in X\cap Z$ such that $A_{xy} \neq 0$, let $X' := X\symdiff \{x,y\}$, let $Y' := Y\symdiff \{x,y\}$, $A' := A^{xy}$, and let $Z' := Z\setminus \{x,y\}$. Since $A^{xy}-u = (A-u)^{xy}$, the matrix $A'-u$ is a $\parf$-matrix and $M\delete u = M[I\  (A'-u)]$. Likewise $A'-v$ is a $\parf$-matrix and $M\delete v = M[I\  (A'-v)]$.
  \begin{claim}
    $Z'$ incriminates $(M,A')$.
  \end{claim}
  \begin{subproof}
    Note that $\det(A'[Z']) = \pm A_{xy}^{-1} \det(A[Z])$. Therefore, if $\det(A[Z]) \not \in \parf$, then certainly $\det(A'[Z']) \not \in \parf$ and the claim follows. Otherwise, observe that $X'\symdiff Z' = X \symdiff Z$, so $X'\symdiff Z'$ is a basis of $M$ if and only if $X\symdiff Z$ is a basis. Moreover, $\det(A'[Z']) = 0$ if and only if $\det(A[Z]) = 0$. The claim now follows from Definition \ref{def:incrimi}.
  \end{subproof}
  But $Z' \cap Y' = (Z\cap Y) \setminus y$, contradicting minimality of $|Z\cap Y|$.
\end{proof}

For the remainder of this section we assume $A-u$, $A-v$ are $\parf$-matrices, $M\delete u = M[I\  (A-u)]$, $M\delete v = M[I\  (A-v)]$, and $M\delete u, M\delete v \in \matset$. We also assume that $a,b \in X$ are such that $\{a,b,u,v\}$ incriminates $(M,A)$.

Pivots were used to create a small incriminating set, but they may destroy it too. We identify some pivots that don't.

\begin{definition}\label{def:allowable}
  If $x \in X, y \in Y\setminus\{u,v\}$ are such that $A_{xy} \neq 0$, then a pivot over $xy$ is \emph{allowable} if there are $a', b' \in X\symdiff \{x,y\}$ such that $\{a',b',u,v\}$ incriminates $(M,A^{xy})$.
\end{definition}

\begin{lemma}\label{lem:allowableab}
  If $x \in \{a,b\}, y \in Y\setminus\{u,v\}$ are such that $A_{xy} \neq 0$, then $\{a,b,u,v\}\symdiff\{x,y\}$ incriminates $(M,A^{xy})$.
\end{lemma}

\begin{proof}
  By symmetry we may assume $x = a$. Let $Z := \{a,b,u,v\}$ and $Z' := \{y,b,u,v\}$. First suppose $\det(A[Z]) \not\in\parf$, but $\det(A^{ay}[Z'])\in \parf$. Then $A^{ay}[Z\cup y]$ is a $\parf$-matrix. Indeed: all entries are in $\parf$, $\det(A^{ay}[\{y,b,a,u\}])\in\parf$, and $\det(A^{ay}[\{y,b,a,v\}]) \in \parf$. This is clearly impossible, since $(A^{ay})^{ya}$ is scaling-equivalent to $A$, after which Proposition~\ref{prop:pivotproper} implies that $A[Z\cup y]$ is a $\parf$-matrix. Hence $\det(A^{ay}[Z'])\not\in\parf$, and the lemma follows.

  Next suppose that $\det(A[Z]) = 0$ and that $X\symdiff Z$ is a basis of $M$. Consider $M' := M_X[Z\cup y]$. Since $\det(A[Z])\in \parf$, $A[Z\cup y]$ is a $\parf$-matrix. Let $N' := M[I\  A[Z\cup y]]$. We have $N' \neq M'$, since $\{u,v\}$ is a basis of $M'$ yet dependent in $N'$. But since $\{u,v\}$ is dependent in $N'$, we have $\det(A^{ay}[Z']) = 0$. Since $X\symdiff Z = (X\symdiff \{a,y\})\symdiff Z'$, the lemma follows.

  The final case, where $\det(A[Z]) \in \parf^*$ and $B\symdiff Z$ is dependent in $M$, is similar to the second and we omit the proof.
\end{proof}

\begin{lemma}\label{lem:allowablezeros}
  If $x \in X\setminus \{a,b\}, y \in Y\setminus\{u,v\}$ are such that $A_{xy} \neq 0$ and either $A_{xu} = A_{xv} = 0$ or $A_{ay} = A_{by} = 0$, then $\{a,b,u,v\}$ incriminates $(M,A^{xy})$.
\end{lemma}

\begin{proof}
  Let $Z := \{a,b,u,v\}$ and define $X' := X\symdiff\{x,y\}$.
  Since $A^{xy}[Z] = A[Z]$, we have $\det(A^{xy}[Z])\in \parf$ if and only if $\det(A[Z])\in\parf$. Therefore we only need to prove the two cases where $\det(A[Z]) \in \parf$. Define $M' := M_X[Z\cup \{x,y\}]$.

  \begin{claim}
    $x$ and $y$ are either in series or in parallel in $M'$.
  \end{claim}
  \begin{subproof}
    If $A_{ay} = A_{by} = 0$ then $x$ and $y$ are clearly in parallel, since they are in parallel in $M'\delete v = M[I\  A[\{x,a,b,y,u]]$. Now assume $A_{xu} = A_{xv} = 0$. If $x$ and $y$ are not in series, then $\{x,y,z\}$ is a cobasis of $M'$ for some $z \in Z$. Clearly $\{y,u,v\}$ is a cobasis of $M'$, so $\{x,y,u'\}$ is a cobasis of $M'$ for some $u'\in\{u,v\}$. Without loss of generality, assume $u' = u$. But then a pivot over $xv$ should be possible in $M'\delete u = M[I\  A[\{x,a,b,y,v\}]]$, contradicting $A_{xv} = 0$.
  \end{subproof}

  But now it follow that $\{x,u,v\}$ is a basis of $M'$ if and only if $\{y,u,v\}$ is a basis of $M'$, and hence that $X\symdiff Z$ is a basis of $M$ if and only if $X'\symdiff Z$ is a basis of $M$. The lemma follows.
\end{proof}

The next theorem gives sufficient conditions under which a certain minor of $M$ can be shown to be outside $\matset$.

\begin{theorem}\label{thm:incriminatingsetinminor}
  Let $N'$ be a strong stabilizer for $\matset$ and suppose $C \subseteq E(M)$ is such that $M_X[C]$ is strictly $N'$-fragile. If there exist subsets $Z,Z_1, Z_2\subseteq E(M)$ such that
  \begin{enumerate}
    \item $u \in Z_1\setminus Z_2$, $v \in Z_2\setminus Z_1$;
    \item $C\cup \{a,b\}\subseteq Z \subseteq Z_1\cap Z_2$;
    \item $M_X[Z]$ is connected;
    \item $M_X[Z_1]$ is $3$-connected up to series and parallel classes;
    \item $M_X[Z_2]$ is $3$-connected up to series and parallel classes;
    \item $\{a,b,u,v\}$ incriminates $(M_X[Z_1\cup Z_2],A[Z_1\cup Z_2])$;
  \end{enumerate}
  then $M_X[Z_1\cup Z_2]$ is not strongly $\pf$-stabilized by $N'$.
\end{theorem}

\begin{proof}
  Let $C$, $Z_1$, and $Z_2$ be as in the theorem. Suppose that, contrary to the result claimed, $M_X[Z_1\cup Z_2]$ \emph{is} strongly $\parf$-stabilized by $N'$. Then $M_X[Z_1\cup Z_2] = M[I\  A']$, where $A'$ is an $(X\cap (Z_1\cup Z_2)) \times (Y\cap (Z_1\cup Z_2))$ $\parf$-matrix. Since $N'$ is a strong stabilizer for $\matset$, we may assume that $A'$ was chosen so that $A'[C] = A[C]$. By Lemma \ref{lem:sistab} and its dual, then, $A'[Z_1]$ is scaling-equivalent to $A[Z_1]$ and $A'[Z_2]$ is scaling-equivalent to $A[Z_2]$. Since $Z\subseteq Z_1 \cap Z_2$, also $A'[Z\cup u]$ is scaling-equivalent to $A[Z\cup u]$ and $A'[Z\cup v]$ is scaling-equivalent to $A[Z\cup v]$.

  Since $M_X[Z]$ is connected, it follows from Lemma~\ref{lem:delpairscale} that $A'[Z\cup \{u,v\}]$ is scaling-equivalent to $A[Z\cup \{u,v\}]$. But then $\det(A'[\{a,b,u,v\}]) = p \det(A[\{a,b,u,v\}])$ for some $p \in \pf^*$, and hence $\{a,b,u,v\}$ incriminates $(M_X[Z_1\cup Z_2],A')$, a contradiction.
\end{proof}

\section{Excluded minors containing a strong stabilizer}\label{sec:thestrongproof}
The main step in our proof of Theorem \ref{thm:rotapf} is the following result:

\begin{theorem}\label{thm:rotapfN}
  Let $s, t$ be positive integers, let $\parf$ be a finitary partial field, let $\matset$ be a well-closed class of $\pf$-representable matroids, and let $\mathcal{N}$ be a set of $\pf$-representable matroids such that, for each $N' \in \mathcal{N}$,
  \begin{enumerate}
    \item\label{it:pfN1} $N'$ is 3-connected and non-binary;
    \item\label{it:pfN2} $N'$ is a stabilizer for $\matset(\pf)$;
    \item\label{it:pfN3} $N'$ is a strong stabilizer for $\matset$.
  \end{enumerate}
  Let $N \in \mathcal{N}$ be a matroid with the following additional property.
  \begin{enumerate}\addtocounter{enumi}{3}
    \item\label{it:pfN4} If $M'$ is an excluded minor for $\matset$ having an $N$-minor and $M'$ is $\pf$-representable, then either $M'$ is not strongly stabilized by $N$ or $M'$ has branch width at most $s$.
  \end{enumerate}
  If all strictly $\mathcal{N}$-fragile matroids have branch width at most $t$, then there is a constant $l$ depending only on $s,t, \pf, \matset, \mathcal{N},N$, such that an excluded minor $M$ for $\matset$, with $N\minorof M$, has branch width at most $l$.
\end{theorem}

Note that \eqref{it:pfN4} is trivially satisfied if $\matset$ contains all 3-connected $\pf$-representable matroids strongly stabilized by $N$. In the applications in this paper this will always be the case. Moreover, within this paper we will only apply this result with $|\mathcal{N}| = 1$. We expect that the more general version will be useful in other contexts.

The proof can be summarized as follows. First, we pick an excluded minor having an $N$-minor but big branch width, and we select a deletion pair $\{u,v\}$ preserving $N$. We construct a matrix $A$ that is close to representing $M$ and locate a small incriminating set, $\{a,b,u,v\}$. Then we identify a $3$-connected minor $M'$ using $\{a,b,u,v\}$ such that $M'\contract \{a,b\}\delete \{u,v\}$ is $\mathcal{N}$-fragile. Now $\{u,v\}$ may not be a deletion pair for $M'$ since the connectivity of $\co(M'\delete u)$, $\co(M'\delete v)$, $\co(M'\delete \{u,v\})$ may be too low. We count the $1$- and $2$-separations and find that the number does not depend on $\mathcal{N}$ or $\parf$. But then only a constant number of blocking sequences need to be added back to $M'$ to repair the connectivity. The resulting matroid, $M''$ say, has branch width bounded by the branch width of $M'$ plus some constant. But $M''$ still has a strong stabilizer $N' \in \mathcal{N}$ as minor, and we can show $M'' \not \in \matset$, which leads to a contradiction.

\begin{proof} 
	Let $\pf$, $\matset$, $\mathcal{N}$, $N$, $s$, $t$ be as in the theorem. Let $r$ be an integer such that the excluded minors $M$ for $\matset$ with $\min\{\rank(M)-\rank(N), \rank(M^*)-\rank(N^*)\} < 3$ have branch width at most $r$. By Lemmas \ref{lem:exmindual} and \ref{lem:exminrankbound} there are finitely many such $M$, so $r$ exists. Let $l := \max\{r,s,t+4109\}$.

  Suppose that $M$ is an excluded minor for $\matset$ having an $N$-minor, but $\bw(M) > l$. Then $\rank(M)-\rank(N)\geq 3$ and $\rank(M^*)-\rank(N^*) \geq 3$. Let $E$ be the ground set of $M$. By Corollary~\ref{cor:delpairexists}, some $M' \in \{M,M^*\}$ has a deletion pair $\{u,v\}$ such that $M'\delete u$ is $3$-connected. By swapping $N$ with $N^*$ and $M$ with $M^*$ if necessary, we may assume $M' = M$. Pick sets $B, E_N$ such that $B$ is a basis of $M$ and $E_N \subseteq E\setminus \{u,v\}$ is such that $M_B[E_N] \cong N$. 

  By \eqref{it:pfN4} and the fact that $\bw(M) > s$, $M$ is either not $\pf$-representable or $M$ is not strongly stabilized by $N$. In the latter case it follows from \eqref{it:pfN2} that $M$ is stabilized by $N$. So in both cases there must be some representation of $N$ that does not extend to a representation of $M$. Fix an $(E_N\cap B)\times (E_N\setminus B)$ $\pf$-matrix $D$ with $N = M[I\  D]$ such that $D$ does not extend to a representation of $M$, and let $A'$ be the matrix described in Theorem~\ref{thm:uniquematrix}. 

  It follows that some $S\subseteq E$ incriminates $(M,A')$. Clearly $u,v \in S$. By Theorem~\ref{thm:incriminatingsetsmall}, there exists an $X\times Y$ matrix $A$ geometrically equivalent to $A'$ such that $a,b \in X$, $u,v \in Y$, and $\{a,b,u,v\}$ incriminates $(M,A)$. By Proposition \ref{prop:scalepivotcommute}, $A$ is unique up to scaling.

  Let $C\subseteq E\setminus \{u,v\}$ be a smallest possible set such that $M_X[C]$ has a minor isomorphic to a member of $\mathcal{N}$. Since $M\delete \{u,v\}$ has an $N$-minor, $C$ exists.

  \begin{claim}
    $M_X[C]$ is $3$-connected.
  \end{claim}

  \begin{subproof}
    For all $x \in C$, $M_X[C\setminus x]$ has no minor in $\mathcal{N}$. Hence, if $x\in C\cap X$ then $x \not\in \conset_{\mathcal{N},M}$, and if $x \in C\cap Y$ then $x \not \in \delset_{\mathcal{N},M}$. It follows that $M_X[C]$ is strictly $\mathcal{N}$-fragile. Clearly $M_X[C]$ has no loops or coloops. By Proposition~\ref{prop:almostconn}, $M_X[C]$ is $3$-connected up to series and parallel classes. Suppose $M_X[C]$ is not $3$-connected, and let $\{e,f\}$ be a parallel pair. By Lemma~\ref{lem:almostpartition}\eqref{it:parpairdel}, $e,f \in \delset_{\mathcal{N},M}$. Since $X$ is a basis of $M$ and $\rank_M(\{e,f\}) = 1$, $|X\cap \{e,f\}| \leq 1$, say $f \not \in X$. But then $M_X[C\setminus f]$ has a minor in $\mathcal{N}$, a contradiction. The same argument shows that $M_X[C]$ has no series pairs.
  \end{subproof}

  Be aware that $M_X[C]$ may have no $N$-minor. However, it still contains \emph{some} strong stabilizer as minor. Let $N'$ be a minor of $M_X[C]$ such that $N' \in \mathcal{N}$. By our assumptions we have $\bw(M_X[C]) \leq t$.

  We now refine the choice of our small incriminating set. By $d_X(U,W)$ we denote the minimal distance between the vertices indexed by $U$ and the vertices indexed by $W$ in $G(M,X)$.

  \begin{assumption}\label{ass:lexmin}
    $X,a,b,C$ were chosen such that $(d_X(a,C), d_X(b,C))$ is lexicographically minimal.
  \end{assumption}

  We now start constructing sets $Z$, $Z_1$, $Z_2$ having the properties in Theorem~\ref{thm:incriminatingsetinminor}.

  \begin{claim}\label{cl:Z}
    There exists a set $Z \subseteq E\setminus\{u,v\}$, with $C\cup \{a,b\} \subseteq Z$, such that $M_X[Z]$ is connected. Moreover, $Z$ can be chosen so that $|Z| \leq |C|+8$.
  \end{claim}

  \begin{subproof}
    Let $P_a$ be a shortest $a-C$ path in $G(M,X)$. Suppose $|P_a| = k > 3$, say $P_a = (a,x_1, x_2, x_3, \ldots, x_k)$, where $x_k \in C$. Then $x_2$ labels a row of $A$. Also $A_{x_2c} = 0$ for all $c \in C$, and $A_{ax_3} = A_{bx_3} = 0$. It follows that a pivot over $x_2x_3$ is allowable and $A^{x_2x_3}[C] = A[C]$. However, $d_{X\symdiff\{x_2,x_3\}}(a,C) < d_{X}(a,C)$, a contradiction to Assumption~\ref{ass:lexmin}.

    Similarly, if $P_b$ is a shortest $b-(C\cup P_a)$ path, then $|P_b| \leq 3$. Now $M_X[C\cup P_a \cup P_b]$ is connected, and the result follows.
  \end{subproof}

  Let $Z$ be as in Claim~\ref{cl:Z}. Note that $\bw(M_X[Z]) \leq \bw(M_X[C]) + 8$, by Proposition~\ref{prop:bwmonotone}. Since $\{u,v\}$ is a deletion pair, $\co(M\delete v)$ is $3$-connected.

  \begin{claim}\label{cl:contractoutsideC}
    There is a set $S \subseteq (X \setminus Z) \cup \{a,b\}$ such that $M_X[E\setminus (S\cup v)]$ is $3$-connected and isomorphic to $\co(M\delete v)$.
  \end{claim}

  \begin{subproof}
    Let $S_1$ be a series class in $M\delete v$. At most one element of $S_1$ is not in $X$. It follows that we can obtain a matroid isomorphic to $\co(M\delete v)$ by contracting only elements from $X$. Let $S \subset X$ be such that $\co(M\delete v) \cong M\contract S\delete v$, and suppose $S$ was chosen such that $|S\cap (Z\setminus \{a,b\})|$ is minimal. Let $x \in (X \setminus (C\cup \{a,b\}))\cap Z$. Then $x$ is in a shortest $a-C$ path or in a shortest $b-C$ path. In either case $A[x,Y\setminus v]$ has at least two nonzero entries. Likewise, if $x \in X\cap C$ then $A[x,Y\setminus v]$ has at least two nonzero entries, since $M_X[C]$ is $3$-connected. It follows that, if $x \in (Z\setminus \{a,b\})\cap S$, then also $y\in X$ for all $y$ such that $x,y$ are in series. Clearly $y \not \in Z\setminus \{a,b\}$, as $M_X[Z\setminus \{a,b\}]$ has no series classes. There is such a $y$ that is not in $S$. But then $M_X[Z\setminus (S\cup v)] \cong M_X[Z\setminus (S\symdiff\{x,y\}\cup v)]$, contradicting minimality of $|S\cap (Z\setminus\{a,b\})|$.
  \end{subproof}

  Let $S$ be as in Claim~\ref{cl:contractoutsideC}.
  \begin{claim}\label{cl:bridgedelv}
    Let $Z_0' \subseteq E\setminus (v \cup S)$ be such that $(Z\setminus S)\cup u \subseteq Z_0'$ and such that
    $M_X[Z_0']$ has exactly $k$ distinct $2$-separations. Then there exists a set $Z_0 \subseteq E\setminus (v\cup S)$ such that $Z_0 \supseteq Z_0'$, $M_X[Z_0]$ is $3$-connected and such that $\bw(M_X[Z_0])\leq \bw(M_X[Z_0']) + 2 k$.
  \end{claim}

  \begin{subproof}
    The result is obvious if $k = 0$, so we suppose $k > 0$. Since $M_X[Z_0']$ is a minor of the $3$-connected matroid $M\contract S \delete v$, no $2$-separation of $M_X[Z_0']$ is induced. Since each matroid in $\mathcal{N}$ is non-binary, $U_{2,4} \minorof N'$. It then follows from Lemma \ref{lem:uncrossed2sep} that $M_X[Z_0']$ has an uncrossed $2$-separation, say $(W_1,W_2)$. Let $v_1, \ldots, v_t$ be a blocking sequence for $(W_1, W_2)$. By Theorem~\ref{thm:blseqbranchwidth}, $\bw(M_X[Z_0'\cup \{v_1, \ldots, v_t\}]) \leq \bw(M_X[Z_0']) + 2$. By Corollary~\ref{cor:block2sep}, the number of $2$-separations in $M_X[Z_0'\cup\{v_1, \ldots, v_t\}]$ is strictly less than $k$. The result now follows by induction.
  \end{subproof}

  Pick $Z_0' = (Z\setminus S) \cup u$. Then $|Z_0'| - |C| \leq 9$, by Claim \ref{cl:Z}. 
  By Lemma~\ref{lem:2sepbound}, $M_X[Z_0']$ has at most $\sepbound{9}$ distinct $2$-separations. Then Claim \ref{cl:bridgedelv} proves the existence of a set $Z_0\supseteq Z_0'$ such that $M_X[Z_0]$ is $3$-connected and such that $\bw(M_X[Z_0]) \leq \bw(M_X[Z_0']) + 2\cdot \sepbound{9}$.

  Define $Z_1 := Z_0\cup\{a,b\}$. For all $x \in S\cap \{a,b\}$, $Z_0 \cup x$ is either $3$-connected or has a series pair. It follows that $M_X[Z_1]$ is $3$-connected up to series classes. Also, $\bw(M_X[Z_1]) \leq \bw(M_X[Z_0]) + 2$.

  \begin{claim}\label{cl:bridgedelu}
    Let $Z_2' \subseteq E\setminus u$ be such that $Z\cup v \subseteq Z_2'$ and such that $M_X[Z_2']$ has exactly $k$ distinct $2$-separations. Then there exists a set $Z_2\subseteq E\setminus u$ such that $Z_2 \supseteq Z_2'$, $M_X[Z_2]$ is $3$-connected, and $\bw(M_X[(Z_1 \setminus u)\cup Z_2])\leq \bw(M_X[(Z_1\setminus u)\cup Z_2']) + 2 k$.
  \end{claim}

  \begin{subproof}
    The result is obvious if $k = 0$, so we suppose $k > 0$. Since $M_X[Z_2']$ is a minor of the $3$-connected matroid $M\delete u$, no $2$-separation of $M_X[Z_2']$ is induced. Again it follows from Lemma \ref{lem:uncrossed2sep} that $M_X[Z_2']$ has an uncrossed $2$-separation, say $(W_1,W_2)$. If $(W_1,W_2)$ is bridged in $M_X[(Z_1\setminus u)\cup Z_2']$ then we set $T = \emptyset$. Otherwise let $(W_1', W_2')$ be a $2$-separation of $M_X[(Z_1\setminus u)\cup Z_2']$ such that $W_1 \subseteq W_1'$ and $W_2 \subseteq W_2'$. Let $v'_1, \ldots, v'_{p'}$ be a blocking sequence for $(W_1', W_2')$ and set $T := \{v'_1, \ldots, v'_{p'}\}$.

    Now $(W_1, W_2)$ is bridged in $M_X[(Z_1\setminus u)\cup Z_2' \cup T]$, so there is a blocking sequence $v_1, \ldots, v_t$ contained in $Z_1\setminus u \cup T$. By Theorem~\ref{thm:blseqbranchwidth}, $\bw(M_X[(Z_1\setminus u)\cup Z_2'\cup \{v_1, \ldots, v_t\}]) \leq \bw(M_X[(Z_1\setminus u)\cup Z_2' \cup T]) \leq \bw(M_X[(Z_1\setminus u)\cup Z_2']) + 2$. By Corollary~\ref{cor:block2sep}, the number of 2-separations in $M_X[Z_2'\cup\{v_1, \ldots, v_t\}]$ is strictly less than $k$. The result now follows by induction.
  \end{subproof}

  Pick $Z_2' := Z \cup v$. Then $|Z_2'| - |C| \leq 9$, by Claim~\ref{cl:Z}. By Lemma \ref{lem:2sepbound}, $M_X[Z_2']$ has at most $\sepbound{9}$ distinct $2$-separations. Then Claim~\ref{cl:bridgedelu} proves the existence of a set $Z_2\supseteq Z_2'$ such that $M_X[Z_2]$ is $3$-connected and such that $\bw(M_X[Z_1 \cup Z_2]) \leq \bw(M_X[(Z_1\setminus u) \cup Z_2]) + 1 \leq \bw(M_X[(Z_1\setminus u)\cup Z_2']) + 2\cdot \sepbound{9}+1$. 

  It now follows from Theorem~\ref{thm:incriminatingsetinminor} that $M_X[Z_1\cup Z_2]$ is not strongly stabilized by $N'$, and hence $M_X[Z_1\cup Z_2]\not\in\matset$. But $M$ is an excluded minor for $\matset$, so we must have $M = M_X[Z_1\cup Z_2]$. By liberal application of Proposition~\ref{prop:bwmonotone} we can now deduce

   \begin{align}
     \bw(M) & = \bw(M_X[Z_1\cup Z_2])\label{eq:final1}\\
            & \leq \bw(M_X[(Z_1\setminus u)\cup Z_2]) + 1\label{eq:final2}\\
            & \leq \bw(M_X[(Z_1\setminus u)\cup Z_2']) + 2 \cdot \sepbound{9} + 1\label{eq:final3}\\
            & \leq \bw(M_X[Z_1\setminus u]) + 2 \cdot \sepbound{9} + 2\label{eq:final4}\\
            & \leq \bw(M_X[Z_1]) + 2 \cdot \sepbound{9} + 2\label{eq:final5}\\
            & \leq \bw(M_X[Z_0]) + 2 \cdot \sepbound{9} + 4 \label{eq:final6}\\
            & \leq \bw(M_X[Z_0']) + 4 \cdot \sepbound{9} + 4 \label{eq:final7}\\
            & \leq \bw(M_X[Z_0'\setminus u]) + 4\cdot \sepbound{9} + 5 \label{eq:final8}\\
            & \leq \bw(M_X[Z]) + 4 \cdot \sepbound{9} + 5 \label{eq:final9}\\
            & \leq \bw(M_X[C]) + 4 \cdot \sepbound{9} + 13 \label{eq:final10}\\
            & \leq t + 4 \cdot \sepbound{9} + 13\label{eq:final11},
   \end{align}
   where \eqref{eq:final3} follows from Claim~\ref{cl:bridgedelu}, \eqref{eq:final4} holds because $Z_2'\setminus (Z_1\setminus u) = \{v\}$, \eqref{eq:final6} holds because $Z_1\setminus Z_0 \subseteq \{a,b\}$, \eqref{eq:final7} follows from Claim~\ref{cl:bridgedelv}, \eqref{eq:final9} holds because $Z\setminus (Z_0'\setminus u) \subseteq \{a,b\}$, and \eqref{eq:final10} follows from Claim~\ref{cl:Z}. But this contradicts our choice of $M$, and our proof is complete.
\end{proof}

\section{Proof of Theorem \ref{thm:rotapf} and Corollary \ref{cor:infichain}}\label{sec:theproof}

\begin{proof}[Proof of Theorem \ref{thm:rotapf}]
  Let $\pf$ be a finitary partial field and let $\matset$ be a well-closed class of $\pf$-representable matroids, each of which has bounded canopy. Suppose that Theorem \ref{thm:rotapf} is false for a matroid $N$. Then $N$ satisfies all conditions of the theorem, yet occurs in an infinite number of excluded minors for $\matset$. Choose $N$ with as few algebraically inequivalent representations over $\pf$ as possible.

  If $N$ has a unique representation over $\pf$ then $N$ is clearly a strong stabilizer. If we apply Theorem \ref{thm:rotapfN} with $\mathcal{N}=\{N\}$ then we find that there is a constant $l$ such that excluded minors for $\matset$ with an $N$-minor have branch width at most $l$. Then Theorem \ref{thm:rotabw} implies the result.

Therefore $N$ has at least two algebraically inequivalent representations over $\matset$. Let $\matset_N \subseteq \matset$ be the smallest well-closed class containing $N$ and all matroids that are strongly stabilized by $N$. If we apply Theorems \ref{thm:rotapfN} and \ref{thm:rotabw} to $\matset_N$, again with $\mathcal{N} = \{N\}$, then we find that there are finitely many excluded minors for $\matset_N$ having an $N$-minor.

  Let $N'$ be such an excluded minor. Then either $N'$ is also an excluded minor for $\matset$, or $N'\in\matset$ but $N'$ is not strongly stabilized by $N$. Assume the latter holds. We know that $N'$ is stabilized by $N$, so $N'$ must have strictly fewer algebraically inequivalent $\pf$-representations than $N$. Hence, by induction, $N'$ is contained in a finite number of excluded minors for $\matset$. It follows that $N$ is contained in only a finite number of excluded minors for $\matset$, a contradiction.
\end{proof}

A similar argument proves Corollary \ref{cor:infichain}:

\begin{proof}[Proof of Corollary \ref{cor:infichain}]
	Let $\pf$ be a finitary partial field. Suppose the Bounded Canopy Conjecture holds for $\pf$, yet $\pf$ has infinitely many excluded minors. First consider the excluded minors with no $U_{2,4}$-minor. Either this set is empty (i.e. $\matset(\pf)$ contains all binary matroids) or it is $\{F_7, F_7^*\}$ (since matroids with no minor in $\{U_{2,4}, F_7, F_7^*\}$ are regular and hence certainly $\pf$-representable).	Hence infinitely many excluded minors contain $U_{2,4}$.
	
	Now consider the following algorithm. Initially, define $\mathcal{S} :=  \{ U_{2,4}\}$. While $\mathcal{S} \neq \emptyset$, do the following. Take $N \in \mathcal{S}$. Let $\matset_N$ be the smallest well-closed class in $\matset(\pf)$ such that every $\pf$-representable matroid stabilized by $N$ is in $\matset_N$. By Theorem \ref{thm:rotapf}, finitely many excluded minors for $\matset_N$ have an $N$-minor. Let $\{M_1, \ldots, M_k\}$ be these excluded minors, and let $\{M_{i_1}, \ldots, M_{i_l}\}$ be the subset that is representable over $\pf$. By definition of $\matset_N$, none of these is stabilized by $N$. Replace $\mathcal{S}$ by $(\mathcal{S}\setminus \{N\}) \cup \{M_{i_1}, \ldots, M_{i_l}\}$ and continue.
	
	Since $\matset(\pf)$ has infinitely many excluded minors, this algorithm does not terminate. It is now straightforward to extract an infinite chain as in the corollary.
\end{proof}

\section{Applications}\label{sec:examples}
In all examples presented here we will have a strong stabilizer at our disposal, so we can apply Theorem \ref{thm:rotapfN}. An advantage of this is that we only need $N$ to have bounded canopy, which we can actually prove in a few cases.

\subsection{Excluded minors for the classes of near-regular and $\sru$ matroids}
Near-regular matroids were introduced in \cite{Whi95} as  the class of matroids representable over a certain partial field that we  denote here by $\nreg$. It turns out that the class of near-regular matroids is exactly the class of matroids representable over all fields of size at least $3$. These representations can be obtained from partial-field homomorphisms, so $\nreg$ is finitary. We apply Theorem~\ref{thm:rotapfN} to give an alternative proof of the following result:\index{excluded minor}

\begin{theorem}[\citet{HMZ11}]\label{thm:exminnreg}
  The class $\matset(\nreg)$ has a finite number of excluded minors.
\end{theorem}

First we need to find the structure of $U_{2,4}$-fragile matroids.

\begin{lemma}\label{lem:U24fragile}
  Let $M$ be a $3$-connected $U_{2,4}$-fragile matroid that has no minor isomorphic to $U_{2,6}$ or $U_{4,6}$. Then exactly one of the following holds.
  \begin{enumerate}
    \item $M$ has rank or corank two;
    \item $M$ has a minor isomorphic to $F_7^-$ or $(F_7^-)^*$;
    \item $M$ has rank at least 3 and is a whirl.
  \end{enumerate}
\end{lemma}

The proof follows easily from the following result:

\begin{lemma}[\citet{GGK}, Lemma 3.3]\label{lem:notwhirlthingy}
  Let $M$ be a $3$-connected, non-binary matroid that is not a whirl. Then $M$ has a minor in the set
  \begin{align*}
    \{U_{2,5}, U_{3,5}, F_7^-, (F_7^-)^*, P_7, P_7^*, O_7, O_7^*\}.
  \end{align*}
\end{lemma}

\begin{proof}[Proof of Lemma \ref{lem:U24fragile}]
  Suppose that the lemma is false, and let $M$ be a matroid that is not in one of the classes mentioned. Then $M$ must have rank and corank at least $3$. It is easily checked that each of $P_7$, $O_7$, and their duals has an element that is both deletable and contractible, so by Lemma \ref{lem:notwhirlthingy}, $M$ must have a $U_{2,5}$- or $U_{3,5}$-minor.
  
  By the Splitter Theorem, $M$ must have a one-element extension of $U_{n-2,n}$ or a one-element coextension of $U_{2,n}$ as a minor, where $n \geq 5$. It is readily checked that $M$ then has a minor in $P_6, Q_6, U_{3,6}$, each of which has an element that is both deletable and contractible, a contradiction.
\end{proof}

\begin{lemma}\label{lem:nregU24}
  Let $M$ be an excluded minor for $\matset(\nreg)$. If $M \not\in \{F_7,F_7^*\}$, then $M$ has a $U_{2,4}$-minor.
\end{lemma}
\begin{proof}
  It is readily checked that $F_7$ is an excluded minor for $\matset(\nreg)$. But if $M$ has no minor in $\{F_7, F_7^*, U_{2,4}\}$, then $M$ is regular and hence certainly near-regular.
\end{proof}

\begin{lemma}\label{lem:nregalmostU24whirl}
  If $M \in \matset(\nreg)$ is $3$-connected and strictly $U_{2,4}$-fragile, then $M$ is a whirl.
\end{lemma}

\begin{proof}[Proof of Lemma~\ref{lem:nregalmostU24whirl}]
  The matroids $U_{2,5}$, $F_7^-$, and their duals are not near-regular. The result follows from Lemma \ref{lem:U24fragile}.
\end{proof}

\begin{lemma}[\citet{GOVW98}]
  The matroid $U_{2,4}$ is a strong stabilizer for $\matset(\nreg)$.
\end{lemma}

\begin{proof}
  Since $U_{2,4}$ has no near-regular $3$-connected single-element extensions or coextensions, the stabilizer theorem from \cite{Whi96} immediately implies that $U_{2,4}$ is a stabilizer. Since $U_{2,4}$ is uniquely representable over $\nreg$, it is strong.
\end{proof}

\begin{proof}[Proof of Theorem~\ref{thm:exminnreg}]
  Lemma~\ref{lem:nregU24} implies that finitely many excluded minors have no $U_{2,4}$-minor. But $U_{2,4}$ is non-binary, $3$-connected, a strong stabilizer, and has bounded canopy over $\nreg$ (by Lemma~\ref{lem:nregalmostU24whirl} and Lemma~\ref{lem:bwwhirl}). Hence Theorems \ref{thm:rotapfN} and \ref{thm:rotabw} imply that finitely many excluded minors do have a $U_{2,4}$-minor, so the result follows.
\end{proof}

Let $\psru$ be the sixth-roots-of-unity partial field introduced by Whittle \cite{Whi97}. He showed that $\matset(\psru)$ equals the set of matroids representable over both $\GF(3)$ and $\GF(4)$. All results above remain valid if we replace $\nreg$ by $\psru$. Hence we also have the following result by \citet{GGK}:

\begin{theorem}\label{thm:exminsru}
  The class $\matset(\psru)$ has a finite number of excluded minors.
\end{theorem}

\subsection{Excluded minors for the class of quaternary matroids}
Using almost the same arguments as in the previous section we can give an alternative proof of the following result by \citet{GGK}:

\begin{theorem}[\citet{GGK}]\label{thm:exminGF4}
  The class $\matset(\GF(4))$ has a finite number of excluded minors.
\end{theorem}

\begin{lemma}\label{lem:GF4U24}
  Let $M$ be an excluded minor for $\matset(\GF(4))$. Then $M$ has a $U_{2,4}$-minor.
\end{lemma}

\begin{proof}
  If $M$ has no $U_{2,4}$-minor then $M$ is binary and hence certainly $\GF(4)$-representable.
\end{proof}

\begin{lemma}
  The matroid $U_{2,4}$ is a strong stabilizer for $\matset(\GF(4))$.
\end{lemma}
\begin{proof}
  \citet{Whi96b} proved that $U_{2,4}$ is a $\GF(4)$-stabilizer. Since $U_{2,4}$ is uniquely representable over $\GF(4)$ (cf. \citet{Ka88}), it is also strong.
\end{proof}

\begin{proof}[Proof of Theorem~\ref{thm:exminGF4}]
  Lemma~\ref{lem:GF4U24} implies that all excluded minors have a $U_{2,4}$-minor. But $U_{2,4}$ is non-binary, $3$-connected, a strong stabilizer, and has bounded canopy over $\GF(4)$ (by Lemma~\ref{lem:U24fragile}, the fact that $F_7^-$ and $(F_7^-)^*$ themselves are excluded minors for $\matset(\GF(4))$, and Lemma~\ref{lem:bwwhirl}). Hence Theorems \ref{thm:rotapfN} and \ref{thm:rotabw} imply that finitely many excluded minors do have a $U_{2,4}$-minor, so the result follows.
\end{proof}

\section{On Rota's Conjecture for quinary matroids}
We will now prove Theorem \ref{thm:gf5bcc} from the introduction. First we need to deal with certain degenerate cases. We will use the following explicit excluded-minor characterizations:

\begin{theorem}[\citet{Tut65}]
  The excluded minors for the class of regular matroids are $U_{2,4}$, $F_7$, and $F_7^*$.
\end{theorem}

\begin{theorem}[\citet{Bix79,Sey79}]
  The excluded minors for $\matset(\GF(3))$ are $U_{2,5}$, $U_{3,5}$, $F_7$, and $F_7^*$.
\end{theorem}

\begin{theorem}[\citet{HMZ11}]
  The excluded minors for the class of near-regular matroids are $U_{2,5}$, $U_{3,5}$, $F_7$, $F_7^*$, $F_7^-$, $(F_7^-)^*$, $P_8$, $\AG(2,3)\delete e$, $(\AG(2,3)\delete e)^*$, and $\Delta_T(\AG(2,3)\delete e)$.
\end{theorem}

\begin{lemma}\label{lem:noU25U35}
  Conjecture \ref{con:boundcanopy} implies that finitely many excluded minors for $\matset(\GF(5))$ have no minor isomorphic to $U_{2,5}$ and $U_{3,5}$.
\end{lemma}

\begin{proof}
  Let $M$ be an excluded minor for $\matset(\GF(5))$ having no minor isomorphic to $U_{2,5}$ and no minor isomorphic to $U_{3,5}$. It is well-known that $F_7$ and $F_7^*$ are excluded minors for $\matset(\GF(5))$, so assume $M$ does not have a minor isomorphic to these two matroids either. Then $M$ is ternary. The class of matroids representable over both $\GF(3)$ and $\GF(5)$ is the class of dyadic matroids. Hence $M$ is an excluded minor for this class.

  If $M$ has no minor in $\{F_7^-, (F_7^-)^*, P_8, \AG(2,3)\delete e, (\AG(2,3)\delete e)^*, \Delta_T(\AG(2,3)\delete e)\}$ then $M$ is near-regular, and hence certainly quinary. Of this list, only the first three matroids are quinary. But each of these is a stabilizer for the class of dyadic matroids (see \citet{PZ08conf}), so Theorem \ref{thm:rotapf} implies that finitely many excluded minors have these as a minor, provided that Conjecture \ref{con:boundcanopy} is true for $\GF(3)$ or for $\GF(5)$.
\end{proof}

\begin{proof}[Proof of Theorem~\ref{thm:gf5bcc}]
  Suppose Conjecture~\ref{con:boundcanopy} holds for $\GF(5)$. By Lemma \ref{lem:noU25U35} all but finitely many excluded minors for $\matset(\GF(5))$ have no minor isomorphic to $U_{2,5}$.

  Now $U_{2,5}$ is a stabilizer for $\matset(\GF(5))$ (see \citet{Whi96b}), so finitely many excluded minors for $\matset(\GF(5))$ have a $U_{2,5}$-minor, by Theorem~\ref{thm:rotapf}. This concludes the proof.
\end{proof}

\paragraph{Acknowledgements.} We thank the anonymous referee for meticulously reading the manuscript, and for helpful suggestions to improve the clarity.

\renewcommand{\Dutchvon}[2]{#1}
\bibliography{matbib2009}
\bibliographystyle{svzthesis} 
\end{document}